\documentclass[reqno,12pt]{amsart}
\usepackage[letterpaper,top=2.5cm,bottom=2cm,left=3cm,right=3cm,marginparwidth=1.75cm]{geometry}
\usepackage[dvipsnames]{xcolor}
\usepackage{tabularx} 
\usepackage{amsmath}
\usepackage{graphicx} 
\usepackage{amsmath}
\usepackage{mathtools} 
\usepackage{amsfonts}
\usepackage{amssymb}
\usepackage{amsthm}
\usepackage[final]{hyperref}
\usepackage{marvosym}
\usepackage{mathrsfs}
\usepackage{enumitem}
\usepackage{dsfont}
\usepackage{tikz-cd} 
\usepackage{comment}
\usepackage[utf8]{inputenc}
\usepackage{changepage}

\usepackage[format=plain, font=footnotesize]{caption}
\usepackage[sort]{cite}
\usepackage{xcolor}

\hypersetup{
	colorlinks=true,       
	linkcolor=teal,        
	citecolor=orange,        
	filecolor=magenta,     
	urlcolor=blue         
}
\usepackage{blindtext}

\newcommand{\pl}{\textbf{pl}}

\newcommand{\Ca}{\textit{\textbf{C}}}

\newcommand{\Gr}{\mathrm{Gr}}
\newcommand{\PP}{\mathbb{P}}
\newcommand{\RR}{\mathbb{R}}

\newcommand{\CC}{\mathbb{C}}

\newcommand{\QQ}{\mathbb{Q}}

\usepackage[capitalise, noabbrev, nameinlink]{cleveref}

\theoremstyle{plain}

\newtheorem{theorem}{Theorem}[section] 
\newtheorem{proposition}[theorem]{Proposition}
\newtheorem{lemma}[theorem]{Lemma}
\newtheorem{conjecture}[theorem]{Conjecture}

\newtheorem{introtheorem}{Theorem} 

\newtheorem{introconjecture}[introtheorem]{Conjecture}

\theoremstyle{definition}
 
\theoremstyle{plain}
\newtheorem{example}[theorem]{Example}
\newtheorem{remark}[theorem]{Remark}

\usepackage{listings}

\definecolor{codedarkgreen}{RGB}{51, 133, 4}
\definecolor{codemaroon}{RGB}{133, 5, 63}
\definecolor{codeteal}{RGB}{0, 128, 96}
\lstdefinelanguage{Macaulay2}{
basicstyle=\small\ttfamily,
alsoletter=",
classoffset=1,
keywords={matrix,minors,gb,transpose,det,ideal,apply,subsets,ker,gens,fold,flatten,entries},
keywordstyle={\color{blue}},
classoffset=2,
morekeywords={from, to, list},
keywordstyle={\color{codemaroon}},
classoffset=3,
morekeywords={QQ},
keywordstyle={\color{codedarkgreen}},
classoffset=4,
morekeywords={MonomialOrder},
keywordstyle={\color{codeteal}},
xleftmargin=1.5cm,
xrightmargin=1em,
columns=fullflexible,
keepspaces=true,
stepnumber=1,
numbers=none,
captionpos=b,
showspaces=false,
frame=none
}

\usepackage[dvipsnames]{xcolor}
\usepackage[ruled,vlined]{algorithm2e}
\usepackage{protosem}

\usepackage{ marvosym }
\usepackage{enumitem}
\usepackage{nicematrix}
\usepackage{ulem}

\usepackage{listings}
 
\definecolor{codedarkgreen}{RGB}{51, 133, 4}
\definecolor{codemaroon}{RGB}{133, 5, 63}
\definecolor{codeteal}{RGB}{0, 128, 96}
\lstdefinelanguage{Macaulay2}{
basicstyle=\small\ttfamily,
alsoletter=",
classoffset=1,
keywords={matrix,minors,gb,transpose,det,ideal,apply,subsets,ker,gens,fold,flatten,entries},
keywordstyle={\color{blue}},
classoffset=2,
morekeywords={from, to, list},
keywordstyle={\color{codemaroon}},
classoffset=3,
morekeywords={QQ},
keywordstyle={\color{codedarkgreen}},
classoffset=4,
morekeywords={MonomialOrder},
keywordstyle={\color{codeteal}},
xleftmargin=1.5cm,
xrightmargin=1em,
columns=fullflexible,
keepspaces=true,
stepnumber=1,
numbers=none,
captionpos=b,
showspaces=false,
frame=none
}

\crefname{listing}{Code}{Codes}
\Crefname{listing}{Code}{Codes}
\crefname{lstlisting}{Code}{Codes}
\Crefname{lstlisting}{Code}{Codes}

\begin{document}

\title{Conic Multiview Varieties}

\thispagestyle{empty}
\date{}

\author{Felix Rydell, Isak Sundelius}

\begin{abstract} We present an algebraic study of projections of conics in space onto multiple images of pinhole cameras. We study the Zariski closures of the projection images, which we call conic multiview varieties. Extending previous work for point and line multiview varieties, we use back-projected cones to describe these varieties. For two views, we determine the defining ideals of conic multiview varieties and describe their singular loci. For any number of views, we characterize when the simplest possible set-theoretic description is achieved based on the geometry of the camera centers. Finally, we investigate the complexity of the associated triangulation problem and conjecture the Euclidean distance degree for the conic multiview variety for two cameras.
\end{abstract}

\maketitle

\setcounter{tocdepth}{2}

\pagenumbering{arabic}



\section*{Introduction} 

A \textit{pinhole camera} is a linear rational map
\begin{align}\begin{aligned}
    \mathbb P^3 &\dashrightarrow \mathbb P^2,\\
    X&\mapsto CX.
\end{aligned}
\end{align}
The \textit{camera matrix} $C\in\mathbb C^{3\times 4}$ is assumed to be full rank. This map is well-defined away from the projective point $c=\ker C$, known as the \textit{center}. Given a \textit{camera arrangement} $\Ca=(C_1,\ldots,C_n)$ of $n$ camera matrices, we get the \textit{joint camera} map
\begin{align}\begin{aligned}\label{eq: point proj}
    \Phi_{\Ca}:\mathbb P^3 &\dashrightarrow (\mathbb P^2)^n,\\
    X&\mapsto (C_1X,\ldots,C_nX).
\end{aligned}
\end{align}
The Zariski closure of the image $\mathrm{Im}\; \Phi_\Ca$ is denoted $\mathcal M_\Ca$ and is called a \textit{multiview variety}. We think of it as the set of possible simultaneous pictures that can be taken of a point with $\Ca$. These varieties are fundamental objects in the \textit{Structure-from-Motion} pipeline, which aims to build 3D models based on 2D images \cite{Hartley2004,kileel2022snapshot}. \textit{Algebraic vision} is the symbiosis of algebraic geometry and computer vision. The algebraic vision community has worked on both set-theoretic and ideal-theoretic properties of $\mathcal M_\Ca$ in great detail \cite{aholt2013hilbert,trager2015joint,agarwal2019ideals}. Optimization properties related to these varieties have also been studied, such as the Sampson error \cite{rydell2024revisiting} and Euclidean distance degrees \cite{harris2018chern,EDDegree_point}. The \textit{Euclidean distance degree} of a variety is the number of complex critical points to the nearest point problem given generic data \cite{draisma2016euclidean}.

An initial step in the Structure-from-Motion pipeline is to extract matching features from the data, called \textit{correspondences}. Classical approaches include SIFT~\cite{lowe2004distinctive}. Recently, new neural network-based methods have emerged, such as SuperGlue \cite{sarlin2020superglue}, LightGlue \cite{lindenberger2023lightglue}, and GlueStick \cite{pautrat2023gluestick}. Point and line features are commonly matched across images. However, other image features such as conics or higher degree plane curves can also be matched~\cite{pollard1985pmf,gdalyahu1999flexible}. As with lines, curves are global features and can therefore be detected and matched with fewer outliers than point features. One prominent source of conic image features is rolling-shutter imaging: rolling-shutter cameras can map lines to conics or higher-degree curves~\cite{hahn2024order}. Conics have also been used to approximate non-linear features \cite{forssen2005view}. 

This paper studies projections of conics in space onto camera planes from an algebraic perspective. Our contributions are threefold: (1) we develop a general projection framework for plane curves; (2) we determine the defining ideal and singular locus in the two-view conic case; and (3) we characterize the $n$-view geometry of conics. We now summarize the organization and main results of the paper.

\Cref{s: Pre,s: Exp Image Maps} construct an explicit rational map $\Phi_{\Ca,2}^{\mathrm{Ch}}$ that sends a space conic $\beta$, represented by its Chow form, to the tuple $\gamma$ of plane conics obtained by projecting $\beta$ onto the image planes of the cameras in the arrangement $\Ca$. Although our main results concern conics, the projection framework applies more broadly to algebraic curves: we treat degree-$d$ plane curves and briefly discuss twisted cubics.

We call the Zariski closure $\mathcal C_{\Ca,2}$ of the image $\mathrm{Im}\; \Phi_{\Ca,2}^{\mathrm{Ch}}$ a \textit{conic multiview variety}. These irreducible varieties are eight-dimensional whenever the camera arrangement contains at least two cameras of distinct centers (\Cref{prop: dim}). 

Given a camera $C$ and an image curve $\gamma$, the \textit{back-projected cone} $B(\gamma)$ is the set of space points $X$ that project onto $\gamma$ via $C$. \Cref{s: Cones} explores the geometry of these back-projected features, which are illustrated in \Cref{fig: conics,fig: twisted}. \Cref{s: Conic MV} studies set-theoretic descriptions for the conic multiview variety. Letting $I$ denote the $3\times 3$ identity matrix, our first main result is the following.

\begin{figure}
    \begin{center}
\begin{tikzpicture}[scale=0.28]
\draw[fill] (1,5) circle (4pt) node[below] {${c_1}$};
\draw[fill] (11.5,1) circle (4pt) node[below] {${c_2}$};
\draw[fill] (22,5) circle (4pt) node[below] {${c_3}$};
\draw (1,8) -- (1,14);
\draw (1,8) -- (6,5);
\draw (6,5) -- (6,11);
\draw (1,14) -- (6,11)  node[above=8] {$\;\;B_1$};
\draw (9,3) -- (9,9);
\draw (9,3) -- (14,3);
\draw (14,3) -- (14,9);
\draw (14,9) -- (9,9) node[above=8, midway] {${B_2}$};
\draw (17,4) -- (17,10);
\draw (17,4) -- (22,7);
\draw (22,7) -- (22,13);
\draw (17,10) -- (22,13) node[left=30] {$B_3$};
\draw[line width=1.5pt, violet,rotate around={-35:(3.5,8.25)}] (3.5,8.25) ellipse (0.9cm and 0.45cm) node[black, above=0.15cm] {${\gamma_1}$};
\draw[line width=1.5pt, violet,rotate around={10:(11.6,4.7)}] (11.6,4.7) ellipse (1cm and 0.4cm) node[black, above=0.15cm] {${\gamma_2}$};
\draw[line width=1.5pt, violet,rotate around={30:(19.55,8.5)}] (19.55,8.5) ellipse (1.1cm and 0.44cm) node[black, above=0.15cm] {${\gamma_3}$};

\draw (1,5) -- (2.6,8.5);
\draw[dotted] (2.6,8.5) -- (4.2, 12);
\draw (4.2,12) -- (7,18);
\draw (1,5) -- (1+15.5*0.2,5+13.26*0.2);
\draw[dotted] (1+15.5*0.2,5+13.26*0.2) -- (1+15.5*0.32,5+13.26*0.32);
\draw (1+15.5*0.32,5+13.26*0.32) -- (1+15.5*0.49,5+13.26*0.49);
\draw[dotted] (1+15.5*0.49,5+13.26*0.49) -- (1+15.5*1,5+13.26*1);

\draw (11.5,1) -- (11.5-4.4*0.215,1+16.6*0.215);
\draw[dotted] (11.5-4.4*0.215,1+16.6*0.215) -- (11.5-4.4*0.48,1+16.6*0.48);
\draw (11.5-4.4*0.48,1+16.6*0.48) -- (11.5-4.4*0.64,1+16.6*0.64);
\draw[dotted] (11.5-4.4*0.64,1+16.6*0.64) -- (11.5-4.4*1,1+16.6*1);
\draw (11.5,1) -- (11.5+5.5*0.21,1+18*0.21);
\draw[dotted] (11.5+5.5*0.21,1+18*0.21) -- (11.5+5.5*0.445,1+18*0.445);
\draw (11.5+5.5*0.445,1+18*0.445) -- (11.5+5.5*0.565-0.05,1+18*0.565);
\draw[dotted] (11.5+5.5*0.565,1+18*0.565) -- (11.5+5.5*1-0.05,1+18*1);

\draw(22,5) -- (22-14.6*0.24,5+12.1*0.24);
\draw[dotted](22-14.6*0.24,5+12.1*0.24) -- (22-14.6*0.34,5+12.1*0.34);
\draw(22-14.6*0.34,5+12.1*0.34) -- (22-14.6*0.51,5+12.1*0.51);
\draw[dotted](22-14.6*0.51,5+12.1*0.51) -- (22-14.6*1,5+12.1*1);

\draw(22,5) -- (22-5*0.28,5+14.4*0.28);
\draw[dotted](22-5*0.28,5+14.4*0.28) -- (22-5*0.46,5+14.4*0.46);
\draw(22-5*0.46,5+14.4*0.46) -- (22-5*1,5+14.4*1);

\fill[ teal!50, opacity=0.3] (1,8) -- (6,5) -- (6,11) -- (1,14) -- cycle;
\fill[ teal!50, opacity=0.3] (9,3) -- (14,3) -- (14,9) -- (9,9) -- cycle;
\fill[ teal!50, opacity=0.3] (17,4) -- (22,7) -- (22,13) -- (17,10) -- cycle;

\draw[line width=1.5pt, violet,rotate around={10:(12,18.5)}] (12,18.5) ellipse (5cm and 2cm) node[black, above=-0.1cm] {$K$};
\end{tikzpicture}
  \hspace{0.5cm}
\scalebox{1}[1]{\includegraphics[width = 0.43\textwidth]{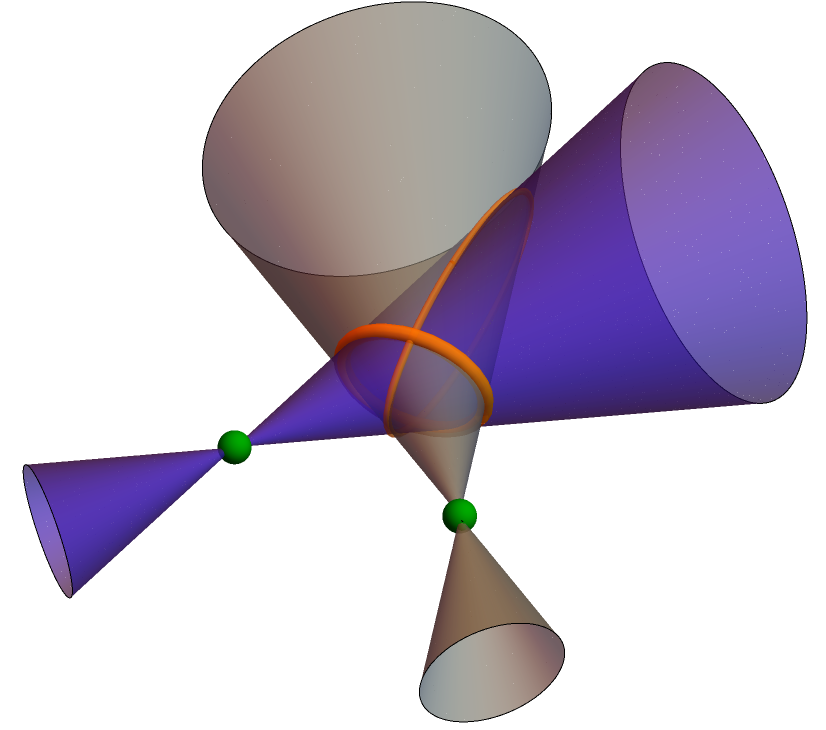}}
\end{center}
    \caption{Left: A conic curve $K$ is projected onto three camera planes, leaving three conic image curves. The back-projected cones, denoted by $B_i$, are the sets of points in space that project onto the corresponding conic. Right: Two back-projected cones through two green camera centers. The cones intersect in a degree-four curve, in this case in the union of two irreducible conics. The second conic is drawn in thin orange.}
    \label{fig: conics} 
\end{figure}


\begin{figure}
    \centering
\includegraphics[width = 0.45\textwidth]{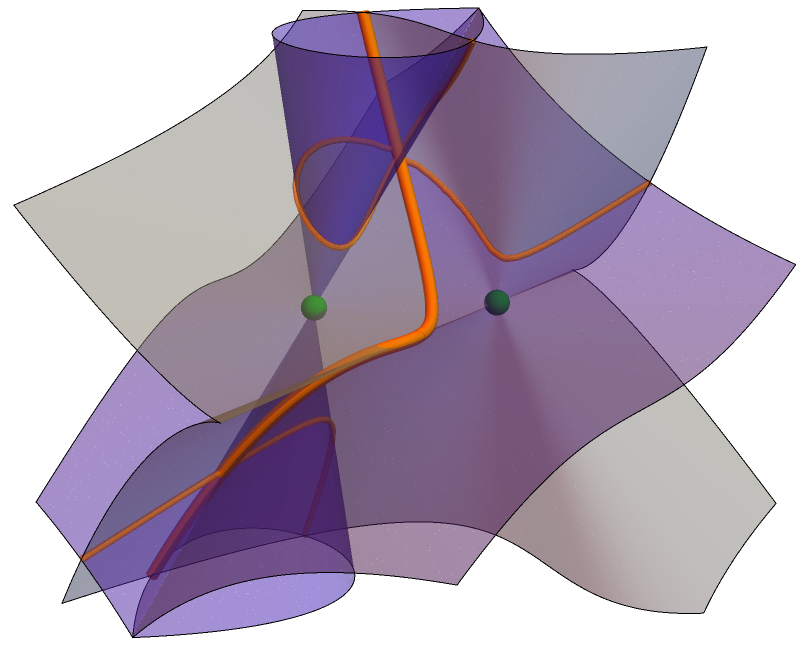}
\hspace{0.5cm}
\begin{tikzpicture}[scale=0.45]
\draw[white,fill] (8,8) circle (4pt) node[below] {};
\draw (6.75,15) -- (6.75,9);
\draw (6.75,9) --  (11.75,8);
\draw (11.75,8) -- (11.75,14);
\draw (11.75,14)-- (6.75,15);
\draw (15.5,16) -- (15.5,10);
\draw  (15.5,10) -- (20.5,12);
\draw (20.5,12) -- (20.5,18);
\draw (20.5,18) -- (15.5,16);

\draw[line width=1.5pt, scale=1, domain=-1.2:0, smooth, variable=\x, violet,rotate around={-80:(22-10*0.375,12+8.6*0.375)}] plot ({1.4*\x*\x+22-10*0.375}, {1.2*\x*\x*\x+12+8.6*0.375});
\draw[line width=1.4pt, scale=1, domain=0:0.2, smooth, variable=\x, violet,rotate around={-80:(22-10*0.375,12+8.6*0.375)}] plot ({1.4*\x*\x+22-10*0.375}, {1.2*\x*\x*\x+12+8.6*0.375});
\draw[line width=1.3pt, scale=1, domain=0.2:0.4, smooth, variable=\x, violet,rotate around={-80:(22-10*0.375,12+8.6*0.375)}] plot ({1.4*\x*\x+22-10*0.375}, {1.2*\x*\x*\x+12+8.6*0.375});
\draw[line width=1.2pt, scale=1, domain=0.4:0.6, smooth, variable=\x, violet,rotate around={-80:(22-10*0.375,12+8.6*0.375)}] plot ({1.4*\x*\x+22-10*0.375}, {1.2*\x*\x*\x+12+8.6*0.375});
\draw[line width=1.1pt, scale=1, domain=0.6:0.8, smooth, variable=\x, violet,rotate around={-80:(22-10*0.375,12+8.6*0.375)}] plot ({1.4*\x*\x+22-10*0.375}, {1.2*\x*\x*\x+12+8.6*0.375});
\draw[line width=0.9pt, scale=1, domain=0.8:0.9, smooth, variable=\x, violet,rotate around={-80:(22-10*0.375,12+8.6*0.375)}] plot ({1.4*\x*\x+22-10*0.375}, {1.2*\x*\x*\x+12+8.6*0.375});
\draw[line width=0.8pt, scale=1, domain=0.9:1, smooth, variable=\x, violet,rotate around={-80:(22-10*0.375,12+8.6*0.375)}] plot ({1.4*\x*\x+22-10*0.375}, {1.2*\x*\x*\x+12+8.6*0.375});

\draw[line width=1.5pt, scale=1, domain=-1.45:0, smooth, variable=\x, violet,rotate around={-90:(7+5*0.45,6+12.5*0.45)}] plot ({1*(\x*\x-1)+7+5*0.45}, {1*(\x*\x*\x-\x)+6+12.5*0.45});
\draw[line width=1.4pt, scale=1, domain=0:0.2, smooth, variable=\x, violet,rotate around={-90:(7+5*0.45,6+12.5*0.45)}] plot ({1*(\x*\x-1)+7+5*0.45}, {1*(\x*\x*\x-\x)+6+12.5*0.45});
\draw[line width=1.3pt, scale=1, domain=0.2:0.4, smooth, variable=\x, violet,rotate around={-90:(7+5*0.45,6+12.5*0.45)}] plot ({1*(\x*\x-1)+7+5*0.45}, {1*(\x*\x*\x-\x)+6+12.5*0.45});
\draw[line width=1.2pt, scale=1, domain=0.4:0.6, smooth, variable=\x, violet,rotate around={-90:(7+5*0.45,6+12.5*0.45)}] plot ({1*(\x*\x-1)+7+5*0.45}, {1*(\x*\x*\x-\x)+6+12.5*0.45});
\draw[line width=1.1pt, scale=1, domain=0.6:0.8, smooth, variable=\x, violet,rotate around={-90:(7+5*0.45,6+12.5*0.45)}] plot ({1*(\x*\x-1)+7+5*0.45}, {1*(\x*\x*\x-\x)+6+12.5*0.45});
\draw[line width=1pt, scale=1, domain=0.8:1, smooth, variable=\x, violet,rotate around={-90:(7+5*0.45,6+12.5*0.45)}] plot ({1*(\x*\x-1)+7+5*0.45}, {1*(\x*\x*\x-\x)+6+12.5*0.45});
\draw[line width=0.8pt, scale=1, domain=1:1.35, smooth, variable=\x, violet,rotate around={-90:(7+5*0.45,6+12.5*0.45)}] plot ({1*(\x*\x-1)+7+5*0.45}, {1*(\x*\x*\x-\x)+6+12.5*0.45});

\fill[ teal!50, opacity=0.3] (6.75,15) -- (6.75,9) -- (11.75,8) -- (11.75,14) -- cycle;
\fill[ teal!50, opacity=0.3] (15.5,16) -- (15.5,10) -- (20.5,12) -- (20.5,18) -- cycle;

\draw[line width=1.5pt, scale=1, domain=-1.45:0, smooth, variable=\x, violet,rotate around={-80:(12,18.5)}] plot ({2*(\x*\x-1)+12}, {2*(\x*\x*\x-\x)+18.5});
\draw[line width=1.4pt, scale=1, domain=0:0.2, smooth, variable=\x, violet,rotate around={-80:(12,18.5)}] plot ({2*(\x*\x-1)+12}, {2*(\x*\x*\x-\x)+18.5});
\draw[line width=1.3pt, scale=1, domain=0.2:0.4, smooth, variable=\x, violet,rotate around={-80:(12,18.5)}] plot ({2*(\x*\x-1)+12}, {2*(\x*\x*\x-\x)+18.5});
\draw[line width=1.2pt, scale=1, domain=0.4:0.6, smooth, variable=\x, violet,rotate around={-80:(12,18.5)}] plot ({2*(\x*\x-1)+12}, {2*(\x*\x*\x-\x)+18.5});
\draw[line width=1.1pt, scale=1, domain=0.6:0.8, smooth, variable=\x, violet,rotate around={-80:(12,18.5)}] plot ({2*(\x*\x-1)+12}, {2*(\x*\x*\x-\x)+18.5});
\draw[line width=1pt, scale=1, domain=0.8:0.925, smooth, variable=\x, violet,rotate around={-80:(12,18.5)}] plot ({2*(\x*\x-1)+12}, {2*(\x*\x*\x-\x)+18.5});
\draw[line width=0.8pt, scale=1, domain=1.075:1.35, smooth, variable=\x, violet,rotate around={-80:(12,18.5)}] plot ({2*(\x*\x-1)+12}, {2*(\x*\x*\x-\x)+18.5});
\end{tikzpicture}
    \caption{Left: Two back-projected cubic cones through two green camera centers. They meet in the twisted cubic drawn in thick orange and a degree-six curve drawn in thin orange. Right: a twisted cubic is projected onto two camera planes. Depending on the position of the camera, we get different types of singularities in the image curve.}
    \label{fig: twisted}
\end{figure}

\begin{introtheorem}\label{thm: PC MV}  For the camera arrangement $\Ca=(\begin{bmatrix}
    0 & I
\end{bmatrix},\begin{bmatrix}
    I & 0
\end{bmatrix})$, we have that $(\gamma,\delta)\in \mathcal C_{\Ca,2}$ if and only if
\begin{align}\label{eq: two view ideal intro}
     \mathrm{rank} \begin{bmatrix}
        \gamma_4^2-4\gamma_3\gamma_5 & \gamma_2\gamma_4-2\gamma_1\gamma_5 & \gamma_2^2-4\gamma_0\gamma_5\\
        \delta_2^2-4\delta_0\delta_5 & \delta_1\delta_2-2\delta_0\delta_4 & \delta_1^2-4\delta_0\delta_3
    \end{bmatrix}\le 1.
\end{align}
The ideal generated by condition $\eqref{eq: two view ideal intro}$ is the defining ideal of $\mathcal C_{\Ca,2}$. 
\end{introtheorem}
The assumption $\Ca=(\begin{bmatrix}
    0 & I
\end{bmatrix},\begin{bmatrix}
    I & 0
\end{bmatrix})$ is not an essential restriction: for a general two-camera arrangement of distinct centers, the corresponding ideal is obtained by a linear coordinate change. The singular locus of a conic multiview variety given two cameras is of codimension one and consists of four components (\Cref{prop: sing locus}). 

Our main result for $n$-view geometry extends work on line multiview varieties, which are the closures of joint projection maps sending space lines to lines in image planes \cite{breiding2023line,LMI}. By \cite[Theorem 2.5]{breiding2023line}, a line multiview variety is cut out by the condition that the back-projected planes of image tuples meet in a line if and only if all centers are distinct and no four are collinear. The corresponding theorem for conic multiview varieties states that we obtain the simplest possible set-theoretic description precisely when no three centers are collinear and no eight centers lie on a conic:

\begin{introtheorem} The conic multiview variety $\mathcal C_{\Ca,2}$ is cut out by the condition on image tuples $\gamma=(\gamma_1,\ldots,\gamma_n)$ that there is a conic $\beta$ in the intersection of all $B(\gamma_i)$ such that 
\begin{align}\label{eq:consist}
    (\beta,B(\gamma_i))\in \overline{\Gamma}_{c_i} \textnormal{ for every }i=1,\ldots,n
\end{align}
if and only if all centers are distinct, no three of them lie on a line, and no eight of them lie on a conic.  
\end{introtheorem}
Here $\overline{\Gamma}_{c_i}$ denotes the closure of the graph of the rational map sending a space conic to its back-projected cone with vertex $c_i$, equivalently, the blow-up associated with this map. The proof uses structured sequences of conics through seven generic lines, as opposed to lines in smooth quadrics as in the case of line multiview varieties. A key insight of this result is that the condition that the back-projected cones meet in a conic is not enough to set-theoretically describe the multiview variety. The cones must additionally be compatible with the common conic in the sense encoded by $\overline{\Gamma}_{c_i}$. For applied settings, the result indicates that when using conic image features in Structure-from-Motion, one should avoid placing three cameras on a line, and avoid putting eight cameras on a conic.

Finally, \Cref{s: Tri Complexity} motivates the following conjecture through monodromy computations. 
\begin{introconjecture}Given an arrangement of two generic cameras $\Ca$, the Euclidean distance degree of $\mathcal C_{\Ca,2}$ is $538$.  
\end{introconjecture}

The full \texttt{Macaulay2}~\cite{M2} code used for the two-view computations and the \texttt{Julia} \cite{bezanson2017julia} code supporting the Euclidean distance degree conjecture are included as ancillary files with the arXiv version of this paper.

The use of conics and more general curves in multiview geometry has an established history. Early work studied stereo reconstruction, pose estimation, and motion recovery from conic correspondences~\cite{safaee19923d,maybank1992projection,de1993conics,quan1996conic}. 
In particular, Quan derives two-view compatibility conditions from the pencil generated by two back-projected quadratic cones (\Cref{ss: two-view}). Schmid and Zisserman instead compute homographies induced by pairs of image conics, while also considering higher-degree and non-algebraic curves~\cite{schmid2000geometry}. Building on this direction, Kaminski and Shashua develop a multiview framework for general algebraic curves, including reconstruction using Chow forms~\cite{kaminski2000calibration,kaminski2004multiple}. In contrast to these primarily reconstruction and estimation-oriented works, we study the Zariski closure of all compatible image-conic tuples, including its defining ideal, singular locus, and $n$-view set-theoretic geometry.

\bigskip

\paragraph{\textbf{Acknowledgements.}} Rydell was supported by the Knut and Alice Wallenberg Foundation within their WASP (Wallenberg AI, Autonomous
Systems and Software Program) AI/Math initiative. The authors would like to thank Kathlén Kohn for many helpful comments and discussions. 

The three-dimensional illustrations in Figure 1 (right), Figure 2 (left), and Figure 3 were created using \texttt{Mathematica}~\cite{Mathematica}.


\section{Preliminaries}\label{s: Pre} 
\setcounter{equation}{0}
\numberwithin{equation}{section} 

Throughout this paper, we work over the complex numbers. Let $\PP^N$ denote the $N$-dimensional complex projective space. A camera $C:\PP^3\dashrightarrow\PP^2$ is a full rank linear mapping. It is well-defined away from its \textit{center} $c:=\ker C$. Let $(\PP^N)^*$ denote the \textit{dual} of $\PP^N$, whose elements $h$ represent hyperplanes $\{X:h^TX=0\}$ of $\PP^N$. Naturally, $(\PP^N)^*\cong \PP^N$, although a distinction between these spaces is often made for clarity.

\Cref{ss: prod} recalls the Grassmannian of lines in space as well as Plücker and Veronese embeddings. \Cref{ss: Chow} presents \textit{moduli spaces}, geometric spaces whose points represent varieties, for plane curves, and derives their Chow forms. Finally, \Cref{ss: epip geo} reviews basic facts of epipolar geometry used in two-view computer vision~\cite{Hartley2004}.

Although our main results concern conics, the projection framework requires little additional work for plane curves of arbitrary degree. We therefore keep $d$ general throughout~\Cref{s: Pre,s: Exp Image Maps}.


\subsection{Plücker and Veronese embeddings}\label{ss: prod}
Every $k$-dimensional subspace of $\PP^N$ is represented by a point of the \textit{Grassmannian} $\Gr(k,\PP^N)$ via the Plücker embedding \cite[Section 8]{Gathmann}. Here, we restrict to the Grassmannian $\Gr(1,\PP^3)$. If $L$ is spanned by $X_0,X_1$, then the image of $L$ under the \textit{Plücker} embedding is the vector $\pl(X_0,X_1)$ of the six $2\times 2$ minors of 
\begin{align}\label{eq: rows of X}
    \begin{bmatrix}
        X_0 & X_1
    \end{bmatrix}^\top.
\end{align}
Letting $[ij]$ denote the minor involving columns $i$ and $j$, we choose the order of these 6 minors to be $[12],[13],[14],[23],[24],[34]$. The rational map $\pl: \left(\PP^3\right)^{2} \dashrightarrow \PP^{5}$ is projectively well-defined precisely when the above matrix is full rank. The image of $\pl$ is $\Gr(1,\PP^3)$.

Let $\mathrm{Sym}_d(\CC^{N+1})$ denote the set of symmetric $(N+1)\times \cdots\times (N+1)$-tensors. It is a vector space of dimension $\binom{d+N}{N}$ and its projectivization is written $\PP(\mathrm{Sym}_d(\CC^{N+1}))$. We define the \textit{Veronese embedding}
\begin{align}\begin{aligned}\label{eq: Veronese}
    \nu_{d}^{N}: \PP^{N}&\to \PP (\mathrm{Sym}_d(\CC^{N+1}))
    ,\\
    X& \mapsto  \otimes^{d} X:= \Big(X_{i_1}\cdots X_{i_{d}}\Big)_{i_1,\ldots,i_{d}=1}^{N+1}.
\end{aligned}
\end{align}
This map is injective and can be viewed as a degree-$d$ morphism from $\PP^N$ to $\PP^{\binom{d+N}{N}-1}$, since $\otimes^{d} X$ is symmetric and has $\binom{d+N}{N}$ independent coordinates. In the special case $d=2$, it sends $X\in \PP^N$ to the symmetric matrix $XX^\top$.

Given a matrix $A\in \PP(\CC^{(h+1)\times(N+1)})$, note that any expression of the form $(AX)_{j_1}\cdots (AX)_{j_{d}}$ is a linear combination of the entries of $\otimes^d X$. Then, from \eqref{eq: Veronese}, it is clear that there is a $\binom{d+h}{h}\times \binom{d+N}{N}$ matrix $\nu_{d}^{h,N}(A)$ such that 
\begin{align}
    \nu_{d}^h(AX)=\nu_{d}^{h,N}(A)\nu_{d}^N(X).
\end{align}

For $A'\in \CC^{(h'+1)\times (h+1)}$ and $A\in \CC^{(h+1)\times (N+1)}$, we have $\nu_{d}^{h',N}(A'A)=\nu_{d}^{h',h}(A')\nu_{d}^{h,N}(A)$. In the case of an invertible square matrix, it follows that $\nu_d^{N,N}(A^{-1})=\nu_d^{N,N}(A)^{-1}$. In general, $\nu_d^{N,h}(A^\dagger)$ is a left pseudo-inverse of $\nu_d^{h,N}(A)$, where $A^\dagger$ is a left pseudo-inverse of $A$. In this paper, a left pseudo-inverse of $A$ is any matrix $A^\dagger$ such that $A^\dagger A=I$. The analogous statement holds for right pseudo-inverses.



\subsection{Chow forms}\label{ss: Chow} To each irreducible variety $\mathcal N$ in $\PP^N$ of dimension $m$, the variety of $(N-m-1)$-dimensional subspaces of $\PP^N$ that meet $\mathcal N$ in a point is a hypersurface of $\Gr(N-m-1,\PP^N)$~\cite{dalbec1995introduction}. In Plücker coordinates, such a hypersurface is defined on the Grassmannian by a single additional homogeneous equation, called the \textit{Chow form}~\cite[Chapter 3, Section 2]{gelfand1994discriminants}.

The set of degree-$d$ plane curves in $\PP^2$ is parametrized by $\alpha\in\PP^{\binom{d+2}{2}-1}$:
\begin{align}
    \alpha\mapsto K_\alpha:=\{x\in \PP^2:\alpha^\top \nu_d^2(x)=0\}.
\end{align}
To parametrize all plane curves in space, we need to vary the supporting hyperplane. Let $h=(h_0:h_1:h_2:h_3)\in(\PP^3)^*$ and consider the two linear maps
\begin{align}
    T(h):=\begin{bmatrix}
     -h_1 & -h_2 & -h_3 \\
     h_0 & 0 & 0 \\
     0 & h_0 & 0 \\
    0 & 0 & h_0 \\
\end{bmatrix},\quad  \widetilde{T}(h):=\begin{bmatrix}
     h_0 &-h_1 & -h_2 & -h_3 \\
     0&h_0 & 0 & 0 \\
     0&0 & h_0 & 0 \\
    0& 0 & 0 & h_0 \\
\end{bmatrix}.
\end{align}
Assume that $h_0\neq 0$. Then $T(h)$ is a linear isomorphism from $\PP^2$ to the plane defined by $h$, and $\widetilde{T}(h)$ is an invertible extension in the following sense: given $x=(x_0:x_1:x_2)\in \PP^2$, denote $(0:x_0:x_1:x_2)\in \PP^3$ by $(0:x)$ and observe that $T(h)x = \widetilde{T}(h)(0:x)$.

Degree-$d$ plane curves in $\{X\in \PP^3:h^\top X=0\}$ are images of degree-$d$ plane curves in $\PP^2$ under the map $T(h)$, i.e., they are of the form $T(h)\cdot K_\alpha :=\{T(h)x: x\in K_\alpha\}$ for some $\alpha$. 

\begin{lemma} For $(h,\alpha)\in (\PP^3)^*\times\PP^{\binom{d+2}{2}-1}$ such that $h_0\neq 0$, the Chow form of the space curve $T(h)\cdot K_\alpha$ is
\begin{align}
    (\nu_d^{2,5}(\widehat{T}(h))^\top\alpha)^\top \nu_d^5(L) =0,
\end{align}
where $L\in \mathrm{Gr(1,\PP^3)}$. 
\end{lemma}
In the proof and for the rest of the paper, we identify the Chow form with the vector that defines the homogeneous equation. In the lemma above, we therefore identify $\nu_d^{2,5}(\widehat{T}(h))^\top\alpha$ as the Chow form of $T(h)\cdot K_\alpha$.

\begin{proof} 
The set of lines meeting $T(h)\cdot K_\alpha$ coincides with the image of the rational map 
\begin{align}\begin{aligned}\label{eq: ch plane}
    \ell_{h,\alpha}:\PP^3\times K_\alpha&\dashrightarrow \Gr(1,\PP^3),\\
    (Y,x)&\mapsto \pl(Y, T(h)x)=\big(\wedge^2\widetilde{T}(h)\big)\pl(\widetilde{T}(h)^{-1}Y,(0:x)).
\end{aligned}
\end{align}
We proceed by first fixing $h=e:=(1:0:0:0)\in (\PP^3)^*$. For this hyperplane $e$, we compute 
\begin{align}
    \ell_{e,\alpha}(Y,x)=(Y_0x_0 : Y_0x_1 : Y_0x_2: Y_1x_1-Y_2x_0 : Y_1x_2-Y_3x_0:Y_2x_2-Y_3x_1).
\end{align}
Assume that $Y_0\neq 0$ and write $L=\ell_{e,\alpha}(Y,x)$. Then $\nu_d^2(x)=\nu_d^2(L_{[0,1,2]})$, where $L_{[0,1,2]}:=(L_0:L_1:L_2)$. By construction, $\alpha^\top \nu_d^2(L_{[0,1,2]})=0$ and this condition cuts out a hypersurface in $\mathrm{Gr}(1,\PP^3)$ (that is irreducible for generic $\alpha$). It follows that $L$ meets $K_\alpha$ if and only if $\alpha^\top\nu_d^2(L_{[0,1,2]})=0$. 

For general $h$, we have by \eqref{eq: ch plane} that $L$ meets $T(h)\cdot K_\alpha$ if and only if 
\begin{align}\label{eq: Chow first}
   \alpha^\top \nu_d^2\Big(\big((\wedge^2 \widetilde{T}(h))^{-1}L\big)_{[0,1,2]}\Big)=0.
\end{align}
The first three rows of $(\wedge^2\widetilde{T}(h))^{-1}$ define the $3\times 6$ matrix
\begin{align}\begin{aligned}
\widehat{T}(h)= \begin{bmatrix}
    h_0 & 0 & 0 &  -h_2 & -h_3 & 0 \\
0&h_0&0& h_1  & 0 & -h_3 \\
0&0&h_0& 0 & h_1 & h_2 \\
\end{bmatrix}.
\end{aligned} 
\end{align}
In particular, $\big((\wedge^2\widetilde{T}(h))^{-1}L\big)_{[0,1,2]}$ equals $\widehat{T}(h)L$. As a consequence, \eqref{eq: Chow first} becomes $\alpha^\top \nu_d^{2,5}(\widehat{T}(h)) \nu_d^5(L)=0$. 
\end{proof}

We write $\mathrm{Ch}_{\mathrm{plane}}(d,\PP^3)\subseteq \PP^{\binom{d+5}{5}-1}$ for the variety of degree-$d$ plane curves in $\PP^3$, represented via their Chow forms.
 In summary, we have established a dominant rational map that sends a hyperplane and a curve in $\PP^2$ to the associated Chow form:
\begin{align}\begin{aligned}\label{eq: iota map}
    \iota_d: (\PP^3)^*\times \PP^{\binom{d+2}{2}-1} &\dashrightarrow \mathrm{Ch}_{\mathrm{plane}}(d,\PP^3),\\
    (h,\alpha)&\mapsto \nu_d^{2,5}(\widehat{T}(h))^\top\alpha.
\end{aligned}
\end{align}
This map is well-defined for $(h,\alpha)$ such that $h_0\neq 0$.
\begin{remark} 
It is clear that the map $\iota_d$ is linear in the second factor $\alpha$. We argue that it is degree-$d$ in the first factor. From the definition of $\widehat{T}(h)$ and the Veronese embedding, the degree in the first factor is at most $d$. To prove that the degree is $d$, we show that there is no cancellation of terms in $\nu_d^{2,5}(\widehat{T}(h))$ that would make the degree less. Let $L_1=(1:0:\cdots:0)\in \PP^5$ and $L_2=(0:\cdots:0:1)\in \PP^5$. We see that
\begin{align}
    \widehat{T}(h)L_1=\begin{bmatrix}
        h_0 \\ 0 \\0
    \end{bmatrix},\quad \widehat{T}(h)L_2=\begin{bmatrix}
        0 \\ -h_3 \\h_2
    \end{bmatrix}.
\end{align}
The Veronese images of these vectors are degree-$d$ in $h$ and share no common factors. By the identity $\nu_d^2(\widehat{T}(h)L_i) =\nu_d^{2,5}(\widehat{T}(h))\nu_d^2(L_i)$, this suffices.     
\end{remark}



\subsection{Epipolar geometry}\label{ss: epip geo} At the core of epipolar geometry lie fundamental matrices and epipoles. In order to define these, let $C_1,C_2$ denote two cameras with distinct centers. The \textit{fundamental matrix} $F^{12}$ is the (projective) $3\times 3$ matrix such that 
\begin{align}\label{eq: det fund}
  y_1^\top F^{12}y_2= \det \begin{bmatrix}
        C_1 & y_1 & 0 \\
        C_2 & 0 & y_2
    \end{bmatrix}.
\end{align}
It encodes the image pairs compatible with the arrangement $\Ca=(C_1,C_2)$. Indeed, $\{(y_1,y_2):y_1^\top F^{12}y_2=0\}$ is the closure of the image of $\Phi_\Ca$ as defined in \eqref{eq: point proj}. The fundamental matrix $F^{12}$ is rank-two, and the kernels $e_2^1:=\ker \;F^{12}$ and $e_1^2:=\ker\;F^{21}$, where $F^{21}:=(F^{12})^\top $, are the \textit{epipoles} of $F^{12}$. Given any $3\times 3$ rank-two matrix $F^{12}$, there is a pair of cameras $C_1$ and $C_2$ such that $F^{12}$ is their fundamental matrix. Such a pair is called a \textit{solution} to $F^{12}$. Motivated by this, we say that a \textit{fundamental matrix} is any $3\times 3$ rank-two matrix. 

In \Cref{s: Conic MV}, we find it convenient to work with cameras of particular forms. The next result shows that given a fundamental matrix $F^{12}$, we can write a solution of cameras $C_1$ and $C_2$ as a function of $F^{12}$ such that the first column of $C_1$ and the last column of $C_2$ is zero. In the result below, the notation $A_{I}$ refers to the submatrix of $A$ we get by keeping only the columns listed in $I=[i_1,\ldots,i_k]$.

\begin{proposition}\label{prop: std sol} Let $F^{12}$ be a fundamental matrix. There is at least one index $i$ such that $(e_2^1)_i\neq 0$. In each case, if we identify the fundamental matrix and epipoles with affine representatives, the following is a solution of cameras:
\begin{align}\begin{aligned}\label{eq: std form}
    i=0\mathrm{)} &\quad C_1= \begin{bmatrix} 0 &
     ([e_1^2]_\times F^{12})_{[2,3]} & e_1^2
\end{bmatrix}, \quad C_2=\left[\begin{smallmatrix}
     & 0 & 0 &0\\
   e_2^1 & 1 & 0&0 \\
    & 0 & 1&0 
\end{smallmatrix}\right],\\
i=1\mathrm{)} &\quad  C_1=\begin{bmatrix} 0 &
     ([e_1^2]_\times F^{12})_{[1,3]} & e_1^2
\end{bmatrix}, \quad C_2=\left[\begin{smallmatrix}
      & 1 & 0&0 \\
    e_2^1 & 0 & 0&0 \\
     & 0 & 1&0
\end{smallmatrix}\right],\\
i=2\mathrm{)}&\quad C_1= \begin{bmatrix} 0 &
     ([e_1^2]_\times F^{12})_{[1,2]} & e_1^2
\end{bmatrix}, \quad C_2=\left[\begin{smallmatrix}
     & 1 & 0 & 0 \\
   e_2^1 & 0 & 1 & 0 \\
    & 0 & 0 & 0
\end{smallmatrix}\right].
\end{aligned}
\end{align}
\end{proposition}
Note that this proposition can be extended so that the $i$-th column of $C_1$ is zero and the $j$-th column of $C_2$ is zero, where $i,j$ are distinct among $\{1,2,3,4\}$. 

\begin{proof} By~\cite[Result 9.14]{Hartley2004}, there is a solution
\begin{align}
    C_1' = \begin{bmatrix}
        [e_1^2]_\times F^{12} & e_1^2
    \end{bmatrix},\quad C_2'= \begin{bmatrix}
        I & 0
    \end{bmatrix}.
\end{align}
If $(e_2^1)_i\neq 0$, let $C_1,C_2$ be the cameras from the corresponding item of~\eqref{eq: std form}. Define $H$ to be the $4\times 4$ matrix obtained from $C_2$ by inserting the bottom row $(0,0,0,1)$. Then, $C_1'H,C_2'H$ coincide with $C_1,C_2$. We are now done, since the fundamental matrix of $C_1'H,C_2'H$ is the same as that of $C_1',C_2'$.
\end{proof}

Finally, we define homographies. Fix two cameras $C_1,C_2$. Given a fixed plane defined by $h\in (\PP^3)^*$ away from both camera centers, $C_1T(h),C_2T(h)$ are invertible and we define 
\begin{align}
    H_{2,1}(h):=(C_2T(h))(C_1T(h))^{-1}.
\end{align}
This invertible matrix is a \textit{homography} from the first camera plane to the second camera plane. For a given point $X$ in the plane defined by $h$, the homography maps the image $C_1X$ in the first camera plane  to the image $C_2X$ in the second.



\section{Explicit Descriptions of Projections of Curves}\label{s: Exp Image Maps}

In this section, we describe explicit projection maps of degree-$d$ plane curves. We assume that $d\ge 2$ from now on. \Cref{ss: plane curves} constructs maps $\Phi_{\Ca,d}^{\mathrm{Ch}}$ and $\Phi_{\Ca,d}$ that fit into the commutative diagram:
\begin{center}
 \begin{tikzpicture}[every node/.style={midway}]
        \matrix[column sep={5.5em,between origins}, row sep={2em}] at (0,0) {
                         \node(D) {}; & \node(E) {$\mathrm{Ch}_{\mathrm{plane}}(d,\PP^3)$}; & \node(F) {}; \\
             \node(A) {$(\PP^3)^*\times \PP^{\binom{d+2}{2}-1}$}; & \node(B) {}; & \node(C) {$\big(\PP^{\binom{d+2}{2}-1}\big)^n$}; \\
        };
        \draw[dashed, ->] (A) -- (E) node[above=0.1cm,right=-0.8cm]{$\iota_d$};
        \draw[dashed, ->] (E) -- (C) node[above=0.1cm,right=0.2cm]{$ \Phi_{\Ca,d}^{\mathrm{Ch}}$};
        \draw[dashed, ->] (A) -- (C) node[above]{$\Phi_{\Ca,d}$};
    \end{tikzpicture}
\end{center}
The Zariski closures of their images coincide. We denote this common closure by $\mathcal C_{\Ca,d}$ and call it a \textit{degree-$d$ plane curve multiview variety}. Multiview varieties model the collection of all possible simultaneous pictures that can be taken of a world feature given the cameras of $\Ca$. Since they are the closures of images of rational maps, they are irreducible.

\begin{proposition}\label{prop: dim} The multiview varieties $\mathcal C_{\Ca,d}$ are irreducible. If $\Ca$ contains at least two cameras with distinct centers, then
\begin{align}
    \dim \mathcal C_{\Ca,d}=\frac{1}{2}d^2+\frac{3}{2}d+3.
\end{align}
\end{proposition}

We leave the proof to \Cref{s: Cones}. \Cref{ss: twisted} briefly describes explicit rational parametrization maps from $\PP^{12}$ to the variety of twisted cubics, the simplest nonplanar projective curves.


\subsection{Plane curves}\label{ss: plane curves} Consider a degree-$d$ plane curve $\beta \in \mathrm{Ch}_{\mathrm{plane}}(d,\PP^3)$. Given a camera $C$ with center $c=(c^{(0)}:c^{(1)}:c^{(2)}:c^{(3)})$, the set of lines through $c$ that meet $\beta$ is
\begin{align}\label{eq: L rest}
    \{L\in \Gr(1,\PP^3): c\in L,\quad \beta^\top\nu_d^5(L)=0\}.
\end{align}
We assume that $\beta$ does not meet $c$, otherwise the above set contains all lines through $c$. The \textit{back-projected line} of a point $x\in \PP^2$ is the line $L$ in $\PP^3$ of points that are projected onto $x$ by $C$. In symbols, $L=c\vee (C^\dagger x)$, where $C^\dagger$ is a pseudo-inverse of $C$ such that $CC^\dagger=I$ and $\vee$ refers to the join of projective subspaces. Let $\gamma$ denote the image of $\beta$ with respect to $C$. The set of lines \eqref{eq: L rest} must equal the set of back-projected lines associated to the curve $\gamma$, which is
\begin{align}\label{eq: back-proj}
    \{c\vee (C^\dagger x)\in \Gr(1,\PP^3): x\in \PP^2,\quad \gamma^\top\nu_d^2(x)=0\}.
\end{align}

The mapping $x\mapsto c\vee (C^\dagger x)$ is linear in $x$ and we write $\widehat{C}$ for the associated matrix. Note that the hat notation differs in meaning based on the context: $\widehat{T}(h)$ is a $3\times 6$ matrix independent of $C$, and $\widehat{C}$ is a $6\times 3$ matrix independent of $h$. It satisfies
\begin{align}\label{eq: E(c) defined}\widehat{C} = E(c) C^\dagger, \quad \textnormal{where}\quad 
    E(c):=\begin{bmatrix}
    c^{(1)} & -c^{(0)} & 0 & 0 \\
    c^{(2)} & 0 & -c^{(0)} & 0 \\
    c^{(3)} & 0 & 0 & -c^{(0)} \\
    0 & c^{(2)} & -c^{(1)} & 0 \\
    0 & c^{(3)} & 0 & -c^{(1)} \\
    0 & 0 & c^{(3)} & -c^{(2)}
\end{bmatrix}.
\end{align}

For \eqref{eq: L rest} and \eqref{eq: back-proj} to be equal, it is clear that we must have $\beta^\top\nu_d^5(\widehat{C}x)=0$ for each $x\in \PP^2$ satisfying $\gamma^\top\nu_d^2(x)=0$ and vice versa:
\begin{align}
    \beta^\top\nu_d^{5,2}(\widehat{C})\nu_d^2(x)=0\quad \textnormal{ if and only if }\quad \gamma^\top\nu_d^2(x)=0.
\end{align}
Hence the image curve $\gamma$ equals $\nu_d^{5,2}(\widehat{C})^\top\beta$ projectively. Motivated by this, we define, for a camera arrangement $\Ca$, the joint image map for plane curves:  
\begin{align}\begin{aligned}
    \Phi_{\Ca,d}^{\mathrm{Ch}}: \mathrm{Ch}_{\mathrm{plane}}(d,\PP^3)&\dashrightarrow \big(\PP^{\binom{d+2}{2}-1}\big)^n,\\
   \beta &\mapsto \big(\nu_{d}^{5,2}(\widehat{C}_1)^\top\beta,\ldots, \nu_{d}^{5,2}(\widehat{C}_n)^\top\beta\big).
\end{aligned}
\end{align}
This map is well-defined precisely for $\beta$ that do not meet any centers. The map $\iota_d$ from \eqref{eq: iota map} allows us to define $ \Phi_{\Ca,d}:=\Phi_{\Ca,d}^{\mathrm{Ch}}\circ \iota_d$ to get an alternative description of the projection of plane curves. Explicitly,
\begin{align}\begin{aligned}\label{eq: PhiCd}
    \Phi_{\Ca,d}: (\PP^3)^* \times \PP^{\binom{d+2}{2}-1}&\dashrightarrow \big(\PP^{\binom{d+2}{2}-1}\big)^n,\\
   (h,\alpha)&\mapsto   \big(\nu_{d}^{5,2}(\widehat{C}_1)^\top \nu_{d}^{2,5}(\widehat{T}(h))^\top\alpha ,\ldots, \nu_{d}^{5,2}(\widehat{C}_n)^\top \nu_{d}^{2,5}(\widehat{T}(h))^\top\alpha \big).
\end{aligned}
\end{align}
The map $\Phi_{\Ca,d}$ has bidegree $(d,1)$. This is in contrast to $\Phi_{\Ca,d}^{\mathrm{Ch}}$, which is a linear map. However, this linearity comes at the cost of replacing the simple product parameter space by the embedded Chow variety.

We next derive $\Phi_{\Ca,d}$ directly, without Chow forms. Take a generic $(h,\alpha)\in (\PP^3)^*\times \PP^{\binom{d+2}{2}-1}$. Its image $\gamma$ is the plane curve of points $CT(h)x$ for $x\in \PP^2$ that satisfy $\alpha^\top\nu_d^2(x)=0$. In symbols, $\gamma$ is characterized by
\begin{align}
    \gamma^\top\nu_d^{2}(CT(h)x)=0\quad \textnormal{ if and only if }\quad \alpha^\top\nu_d^2(x)=0,
\end{align}
and we conclude that $\gamma=\nu_d^{2,2}(CT(h))^{-\top}\alpha$, where $\bullet^{-\top}$ outputs the transpose of the inverse of a matrix. Our next claim is that the projective inverse of $CT(h)$ is precisely $\widehat{T}(h)\widehat{C}$, from which we obtain \eqref{eq: PhiCd}. To prove the claim, fix affine representatives for $h$ and $c$, and consider the $1\times 4$ vector $r(h):= \begin{bmatrix}
    h_0&h_1&h_2&h_3
\end{bmatrix}$. A direct computation gives
\begin{align}
T(h)\widehat{T}(h)E(c)
=
-\big(r(h)c\big)I_4
+
cr(h),
\end{align}
where $r(h)c$ is a scalar and $cr(h)$ is an outer product matrix. With this observation, it is straightforward to conclude that $CT(h)\widehat{T}(h)\widehat{C}$ is the identity up to scaling. 

\begin{remark} In specific scenarios, we may wish to project curves from a fixed plane in $\PP^3$. In this case, the projection simplifies to a linear point projection. Indeed, if the input $h$ is generic and fixed, then the map $\Phi_{\Ca,d}(h,\bullet)$ becomes
\begin{align}\begin{aligned}
    \PP^N&\to (\PP^N)^n,\\
    X&\mapsto (P_1X,\ldots,P_nX),
\end{aligned}
\end{align}
where $N=\binom{d+2}{2}-1$ and $P_i:=\nu_{d}^{2,2}(C_iT(h))^{-\top}$. 
\end{remark}


\subsection{Twisted cubics} \label{ss: twisted} 

Irreducible cubic curves are either plane cubics or twisted cubics~\cite{eisenbud20163264}. The standard twisted cubic curve is the image of the morphism $\varphi: \PP^1\to \PP^3$ sending $(s:t)\mapsto (s^3:s^2t:st^2:t^3)$. A general twisted cubic is a transformation of the standard twisted cubic by a full-rank matrix $M\in \PP(\CC^{4\times 4})$. We denote by $\mathrm{Ch}_{\mathrm{tw}}(\PP^3)\subseteq \PP^{55}$ the (closure of the) set of twisted cubics, represented as Chow forms. Similarly to the case of plane curves, given a camera $C$, the picture taking map from twisted cubics $\mathrm{Ch}_{\mathrm{tw}}(\PP^3)$ to plane cubics $\PP^9$ is given by $\beta \mapsto \nu_3^{5,2}(\widehat{C})^\top \beta$. These image cubics are singular, which follows from the connection between rationality and genus~\cite[Chapter 3, Section 7]{shafarevich1994basic} together with the genus-degree formula~\cite[Section 8.3]{fulton2008algebraic}. Irreducible cubic plane curves have two different types of singularities, nodal and cuspidal. The singularity type of an image curve depends on the pose of the camera with respect to the twisted cubic, as illustrated in \Cref{fig: twisted}. The image curve has a cusp if and only if the center lies on a tangent line to the curve.

The Chow form of the standard twisted cubic is given in dual coordinates in~\cite[Section 1]{dalbec1995introduction}. In primal coordinates, it equals
\begin{align}\label{eq: twist cubic Chow}
    -L_3^3-L_3^2L_2+2L_4L_3L_1-L_5L_1^2-L_4^2L_0+L_5L_3L_0+L_5L_2L_0=0.
\end{align}
Denote by $\omega^{\mathrm{tw}}$ the vector such that \eqref{eq: twist cubic Chow} becomes $(\omega^{\mathrm{tw}})^\top\nu_3^5(L)=0$.  Following the ideas in \Cref{ss: Chow}, we can check that the Chow form of the twisted cubic defined by $M$ is $\nu_3^{5,5}\big((\wedge^2M)^{-1}\big)^\top\omega^{\mathrm{tw}}$. The fibers of this map are, however, three-dimensional, so this construction does not directly yield a generically one-to-one parametrization. 

Still, the space of twisted cubics is rational~\cite{piene1985hilbert}. That is, there is a generically one-to-one map from $\PP^{12}$ that parametrizes it. We summarize two approaches to construct such maps: (1) Any six points in $\PP^3$, no four coplanar, lie on a unique twisted cubic~\cite[Example 1.12 \& Theorem 1.18]{harris2013algebraic}. In particular, pairs of triplets of generic points on fixed distinct planes uniquely determine a twisted cubic. Such sixtuples of points form a rational variety~\cite{mattuck1968field}. (2) Take a generic point $o\in \PP^3$. Each generic twisted cubic has a unique secant line $\ell$ through $o$. The union of the cubic and $\ell$ is the complete intersection of quadrics in a unique pencil. The variety of pencils of quadrics through a line is isomorphic to $\Gr(1,\PP^6)$ and therefore rational, giving rise to a generically one-to-one rational parametrization of twisted cubics.


\section{Back-Projected Cones and Blowups}\label{s: Cones} When studying multiview geometry, it is helpful to work with back-projected features~\cite{ponce2017congruences}. \Cref{ss:back_deg_d} establishes back-projected cones for degree-$d$ curves. \Cref{ss:back_conic} specializes to conic curves and lays the groundwork for \Cref{s: Conic MV}.

\subsection{Back-projected cones}\label{ss:back_deg_d} We define the \textit{back-projected cone} of a degree-$d$ image curve $\gamma$ with respect to the camera $C$ as
\begin{align}\label{eq: back-proj def}
    B_C(\gamma):=\{X\in \PP^3: \gamma^\top\nu_{d}^2 (CX)=\gamma^\top\nu_{d}^{2,3}(C)\nu_{d}^3(X)=0\}
\end{align}
(illustrated in \Cref{fig: conics,fig: twisted}). It is called a cone, since to each point $X$ on $B(\gamma)$ away from the center $c$, the line $c\vee X$ is contained in $B(\gamma)$. We often suppress the index $C$, when the camera is understood from context. Identifying $B(\gamma)$ with the vector $\zeta=\nu_d^{2,3}(C)^\top\gamma \in \PP^{\binom{d+3}{3}-1}$, there is a natural bijection between image curves and back-projected cones:
\begin{align}\begin{aligned}
    \gamma& \mapsto \zeta= \nu_d^{2,3}(C)^\top \gamma,\\
    \zeta &\mapsto \gamma =  \nu_d^{3,2}(C^\dagger)^\top \zeta.
\end{aligned}
\end{align}

\begin{proposition}\label{prop: back-proj} The back-projected cone $B(\gamma)$ is irreducible if and only if $\gamma$ is. The set-theoretic degree of $B(\gamma)$ is the set-theoretic degree of $\gamma$. 
\end{proposition}

\begin{proof} Assume that $B(\gamma)$ is irreducible. The restriction of the camera map $C:\PP^3\dashrightarrow \PP^2$ to $B(\gamma)$ is dominant onto $\gamma$. Since the image of a rational map from an irreducible variety is irreducible, $\gamma$ is irreducible. 

Conversely, assume that $\gamma$ is irreducible. Write $B(\gamma)=B_1\cup \cdots\cup B_k$ for the irreducible components of $B(\gamma)$, and note that these must all be hypersurfaces. The image of $B(\gamma)$ under $C$ is the union of the images $\gamma_i$ of $B_i$. The irreducibility of $\gamma$ implies that $\gamma_i=\gamma$ for some $i$. By definition, $B_i$ equals $B(\gamma)$, proving that $B(\gamma)$ is irreducible.

Let $\gamma_s$ denote the reduced curve associated to $\gamma$ and let $e$ denote its degree. Consider a generic line $\ell\in \PP^2$. It meets $\gamma_s$ in $e$ many points with multiplicity 1. Denote by $L_1,\ldots,L_e$ the corresponding back-projected lines. A generic line in the back-projected plane of $\ell$ meets each $L_i$ in a point, and such a line is generic in $\PP^3$, since $\ell$ is generic. Therefore a generic line meets $B(\gamma)$ in $e$ many points. Set-theoretically, every such intersection point has multiplicity 1, showing the statement.
\end{proof}

\begin{proof}[Proof of \Cref{prop: dim}] It suffices to prove the statement for an arrangement $\Ca=(C_1,C_2)$ of two cameras $C_1$ and $C_2$ with distinct centers $c_1$ and $c_2$. For a fixed $d$, the dimension listed in the statement is the dimension of the domain $(\PP^3)^*\times \PP^{\binom{d+2}{2}-1}$. If the image of $\Phi_{\Ca,d}$ is of less dimension, which we assume by contradiction, then by the fiber dimension theorem~\cite[Chapter 1, Section 8, Corollary 1]{mumford1999red}, the generic fiber would have to be at least one-dimensional. 

For a generic tuple $(\gamma_1,\gamma_2)\in \mathrm{Im}\;\Phi_{\Ca,d}$, both $\gamma_1,\gamma_2$ are irreducible plane curves. Their back-projected cones $B_1,B_2$ are irreducible by \Cref{prop: back-proj}. As we assumed above, there is a one-dimensional family of distinct degree-$d$ curves in $\PP^3$ contained in the $B_1\cap B_2$. This implies that $\dim B_1\cap B_2 = 2$. This is only possible if the cones coincide, by irreducibility. However, this is a contradiction, because irreducible cones of distinct centers are not equal. Indeed, they are unions of lines through the center, and there is only one line that goes through two distinct centers.       
\end{proof}

Given a space curve $\beta$ away from the center $c$ of $C$, the back-projected cone of the image of $\beta$ under $C$ is denoted $c\vee \beta$. It depends only on $c$, as it is the union of lines through $c$ that meet $\beta$ (set-theoretically). 

\begin{lemma} For $\beta\in \mathrm{Ch}_{\mathrm{plane}}(d,\PP^3)$ away from the center $c$ of $C$, we have the identity 
\begin{align}
    c\vee \beta = \nu_d^{5,3}(E(c))^\top\beta.
\end{align}
\end{lemma}

\begin{proof} The image $\gamma$ of $\beta$ under $C$ is $\nu_d^{5,2}(\widehat{C})^\top \beta$. The back-projected cone $B(\gamma)$ is then defined by the vector $\zeta=\nu_d^{2,3}(C)^\top \gamma$, yielding $\zeta= \nu_d^{5,3}(\widehat{C}C)^\top \beta$. 

We are left to show that $\widehat{C}C=E(c)$. With affine representatives, any $X\in \CC^4$ can be written $X =C^\dagger C X + \lambda c$. Since $\ker E(c) = c$, we get the identity $E(c) X = E(c)C^\dagger C X$, proving the claim.   
\end{proof}


\subsection{Blowups}\label{ss:back_conic} 
We study the blowup $\overline{\Gamma}_c$, where \begin{align}\Gamma_c:=\{(\beta,c\vee \beta):\beta \textnormal{ does not meet }c \}\subseteq  \mathrm{Ch}_{\mathrm{plane}}(2,\PP^3)\times \mathcal B_c
\end{align}
and $\mathcal B_{c}$ is the variety of back-projected cones of degree-two plane curves through $c$. This is exactly what we need to later deduce set-theoretic constraints for the multiview variety $\mathcal C_{\Ca,2}$ in \Cref{s: Conic MV}. We characterize the blowup algebraically in \Cref{le: blowup} and geometrically in \Cref{prop: blowup}. Below, a \textit{conic curve} refers to a plane conic curve, i.e., we exclude unions of two lines that are not coplanar.

\begin{lemma}\label{le: blowup}
For $d=2$, we have
\begin{align}\label{eq: blowup}
\overline{\Gamma}_c
=
\left\{
(\beta,B)\ \middle|\
\begin{array}{l}
\beta\subseteq B \text{ set-theoretically};\\[2pt]
\text{for every irreducible component }S\subseteq B_{\mathrm{red}}
\text{ and every line}\\
L\subseteq S \text{ through }c, \text{there exists }Y\in L\cap\beta\text{ such that the}\\\text{tangent line of }\beta\text{ at }Y\text{ is contained in the tangent}\\
\text{plane }T_XS,
\text{ where }X\in L\text{ is general}
\end{array}
\right\}.
\end{align}
\end{lemma}

This statement expresses two conditions separated by a semicolon. 

\begin{proposition}\label{prop: blowup} For $d=2$ and $(\beta,B)\in \overline{\Gamma}_c$, there are three possibilities:
\begin{enumerate}[label=\textnormal{(\Roman*)}]
    \item $c\not \in\beta$. In this case, $B=c\vee \beta$. 
    \item $c\in \beta$, and $\beta$ is either an irreducible conic or the union of two lines, exactly one of which meets $c$. In this case, $B$ is the union of the unique plane containing $\beta$ and a plane containing the tangent line of $\beta$ at $c$.
    \item $\beta$ is the union of two lines that meet at $c$. In this case, $B$ is any cone that contains $\beta$.
\end{enumerate}
\end{proposition}

\begin{proof}[Proof of \Cref{le: blowup}] We first show the containment $\subseteq$. It suffices to show that $\Gamma_c$ lies in the right-hand-side of \eqref{eq: blowup}. Take a $\beta$ away from $c$. Then $(\beta, B=c\vee \beta)$ satisfies the first condition of \eqref{eq: blowup} by construction. The second condition of \eqref{eq: blowup} also follows directly.  

We second show the containment $\supseteq$. We divide this part into three cases:

Case I: $c\not \in\beta$. Then the first condition forces $c\vee \beta\subseteq B$. The second condition ensures that each line in $B$ through $c$ meets $\beta$, so that $B\subseteq c\vee \beta$. Since $B=c\vee \beta$, we have by construction that $(\beta, B)\in \Gamma_c$.

Case II: $c\in\beta$, and $\beta$ is either an irreducible conic or the union of two lines, exactly one of which meets $c$. Let $H$ denote the unique plane containing $\beta$. It follows from the first condition that $H\subseteq B$, and therefore $B=H\cup P$ for some plane $P$. The second condition implies that $P$ contains the tangent line of $\beta$ at $c$. 

We show that $(\beta,B)\in \overline{\Gamma}_c$. If $H=P$, then any sequence $\beta_m\to \beta$ of converging translates of $\beta$ (chosen in an appropriate affine chart) away from $c$ in $H$ satisfies $c\vee\beta_m\to B$. This suffices, since $(\beta_m,c\vee\beta_m)\in \Gamma_c$ for each $m$. If $H\neq P$, fix a generic line $L$ in $P$ through $c$. Denote by $\beta_m\to \beta$ a converging sequence of translates of $\beta$ (chosen in an appropriate affine chart) along $L$. 

The limit of $B_m:=c\vee \beta_m$ contains $H$ because it contains $\beta$, and $L$ because each $\beta_m$ meets $L$. Further, the tangent plane of $B_m$ along the line $L$ is the join of $L$ and the tangent line $T_c\: \beta$ of $\beta$ at $c$, since $\beta_m$ is a translation of $\beta$. Observe that $L\neq T_c\:\beta$, since $L$ is generic, and that the join $L\vee T_c\: \beta$ coincides with $P$. In the limit, then, $B_m$ must contain $P$ as a tangent plane. This is only possible if $B_m\to H\cup P=B$, which suffices.     

Case III: $\beta$ is the union of two lines that meet $c$. There is no well-defined tangent line of $\beta$ at $c$, so the second condition is vacuous. Therefore $(\beta, B)$ lies in the right-hand-side of \eqref{eq: blowup} precisely when $\beta \subseteq B$. Fix such a $B$.

We show that $(\beta,B)\in \overline{\Gamma}_c$. If $B$ is irreducible, then let $H$ be a hyperplane containing $\beta$, and in the case that $\beta$ is a double line and such an $H$ is not unique, choose this $H$ such that it is tangent to $B$. In either case, $H$ is uniquely determined. Take any sequence of hyperplanes $H_m\to H$ away from $c$ and let $\beta_m:=H_m\cap B$. Then $\beta_m$ is away from $c$, $\beta_m\to \beta$, and $c \vee \beta_m= B$ for all $m$. If $B$ is a union $P_1\cup P_2$ of two planes, then let $B_m\to B$ be any converging sequence of irreducible cones containing $\beta$. It suffices to prove that $(\beta, B_m)\in  \overline{\Gamma}_c$ for every $m$, which we did above. 
\end{proof}



\section{Conic Multiview Varieties}\label{s: Conic MV}
In this section, we study algebraic and set-theoretic questions related to \textit{conic multiview varieties} $\mathcal C_{\Ca,2}$. \Cref{ss: two-view} examines two-view geometry, determines the defining ideal of $\mathcal C_{\Ca,2}$ for two cameras, and computes its singular locus. \Cref{ss: three-view} provides examples from three-view geometry. \Cref{ss: n-view geo} characterizes the simplest possible set-theoretic description of conic multiview varieties based on the geometry of the centers.

By the bijection from image curves to back-projected cones, we sometimes identify $\mathcal C_{\Ca,2}$ with the associated variety of back-projected cones. All cones in this section are degree-two, and all conic curves are assumed to be planar.


\subsection{Two-view geometry}\label{ss: two-view}  Before we give the statement, we fix the choice of embedding for the Veronese embedding $\nu_2^2:\PP^2\to \PP^5$ as 
\begin{align}
    (x_0:x_1:x_2)\mapsto (x_0^2:x_0x_1:x_0x_2:x_1^2:x_1x_2:x_2^2).
\end{align}
\begin{theorem}\label{thm: main two view} For the camera arrangement $\Ca=(\begin{bmatrix}
    0 & I
\end{bmatrix},\begin{bmatrix}
    I & 0
\end{bmatrix})$, we have that $(\gamma,\delta)\in \mathcal C_{\Ca,2}$ if and only if
\begin{align}\label{eq: two view ideal}
     \mathrm{rank} \begin{bmatrix}
        \gamma_4^2-4\gamma_3\gamma_5 & \gamma_2\gamma_4-2\gamma_1\gamma_5 & \gamma_2^2-4\gamma_0\gamma_5\\
        \delta_2^2-4\delta_0\delta_5 & \delta_1\delta_2-2\delta_0\delta_4 & \delta_1^2-4\delta_0\delta_3
    \end{bmatrix}\le 1.
\end{align}

The ideal generated by~\eqref{eq: two view ideal} is the defining ideal of $\mathcal C_{\Ca,2}$.
\end{theorem}

\begin{remark}\label{re: lin iso}
For a general camera arrangement $\Ca$, we derive the defining ideal from the theorem as follows. By \Cref{prop: std sol}, we may assume that the cameras are of the form $\widetilde{C}_1=A_1\begin{bmatrix}
    0 & I
\end{bmatrix},\widetilde{C}_2=A_2\begin{bmatrix}
    I & 0
\end{bmatrix}$ for two full-rank $3\times 3$ matrices $A_1,A_2$. The homography $H_{2,1}(h)$ with respect to $(\begin{bmatrix}
    0 & I
\end{bmatrix},\begin{bmatrix}
    I & 0
\end{bmatrix})$ satisfies 
\begin{align}
    H_{2,1}(h)=\begin{bmatrix}
        -h_1 & -h_2 & -h_3\\
        h_0 & 0 & 0\\
        0 & h_0 & 0
    \end{bmatrix}.
\end{align}
The homography between $\widetilde{C}_1,\widetilde{C}_2$ is then
\begin{align}
    A_2H_{2,1}(h)A_1^{-1}.
\end{align}
As shown below, $(\gamma,\delta)\in \mathcal C_{\Ca,2}$ if and only if it lies in the zero locus of~\eqref{eq: two view ideal} after applying the change of coordinates $\gamma\mapsto \nu_2^{2,2}(A_1)^\top\gamma$ and $\delta\mapsto \nu_2^{2,2}(A_2)^\top\delta$.
\end{remark}

\subsubsection{Deriving two-view constraints}

Recall the definition of $T(h)$ from~\Cref{s: Pre}. Given a generic pair $(h,\alpha)\in (\PP^3)^*\times \PP^5$, write $(\gamma,\delta)=\Phi_{\Ca,2}(h,\alpha)$. With the camera arrangement $\Ca=(\begin{bmatrix}
    0 & I
\end{bmatrix},\begin{bmatrix}
    I & 0
\end{bmatrix})$, we list three methods for deriving~\eqref{eq: two view ideal}:

\begin{enumerate}[label=\textnormal{(\Alph*)}]
    \item We know that the back-projected cones $B(\gamma),B(\delta)$ meet in a conic inside $h$. The set of $X\in \PP^3$ lying in $B(\gamma)$, $B(\delta)$ and $h$ is determined by the system
\begin{align}\begin{aligned}
   0&= \gamma_0 X_1^2 + \gamma_1 X_1X_2 +\gamma_2 X_1X_3+\gamma_3 X_2^2+\gamma_4 X_2X_3 +\gamma_5 X_3^2,\\
    0&=\delta_0 X_0^2 + \delta_1 X_0X_1 +\delta_2 X_0X_2+\delta_3 X_1^2+\delta_4 X_1X_2 +\delta_5 X_2^2,\\
    0&=h_0X_0+h_1X_1+h_2X_2+h_3X_3.
\end{aligned}
\end{align}
Since $h_0\neq 0$, we can plug in $X_0=-(h_1X_1+h_2X_2+h_3X_3)/h_0$ into the second equation. We are left with two degree-two equations in three homogeneous variables, at least one of which is irreducible. For the solution set to be positive-dimensional, those two equations must be projectively equal. The resulting system is $\mathcal J_1\subseteq \CC[\gamma,\delta,h]$. Setting $h_0=1$ and eliminating $h_1,h_2,h_3$, we are left with two primary components over $\QQ$. These are $I$, the ideal generated by~\eqref{eq: two view ideal}, and the ideal generated by the single equation
\begin{align}\label{eq: union of planes}
    \delta_2^2\delta_3-\delta_1\delta_2\delta_4+\delta_0\delta_4^2 +\delta_1^2\delta_5-4\delta_0\delta_3\delta_5=0.
\end{align}See \Cref{fig: code} for this computation in \texttt{Macaulay2}.\label{enum:A}
\item 
For each point $x\in \PP^2$ with $\alpha^\top \nu_2^2(x)=0$, we have that
\begin{align}
    \gamma^\top\nu_2^{2}(C_1T(h)x)=0 \quad \textnormal{ and }\quad \delta^\top\nu_2^{2}(C_2T(h)x)=0.
\end{align}
Using the homography $H_{2,1}(h)$, the second expression becomes
\begin{align}
   \delta^\top\nu_2^{2}(H_{2,1}(h))\nu_2^{2}(C_1T(h)x)=0.
\end{align}
Since $C_1T(h)$ is invertible, these equalities imply
\begin{align}\label{eq: homog app}
    \gamma = \nu_2^{2,2}(H_{2,1}(h))^\top \delta.
\end{align}
We write $\mathcal J_2\subseteq \CC[\gamma,\delta,h]$ for the associated polynomial system, which turns out to equal $\mathcal J_1$. 
\label{enum:B}
\item Following Quan~\cite{quan1996conic}, the two back-projected cones
$B(\gamma),B(\delta)$ contain two irreducible conic curves $\beta_1,\beta_2$ (one is given by $(h,\alpha)$). The pencil of quadrics 
\begin{align}
B(\gamma)+\lambda B(\delta)
\end{align}
consists precisely of all quadric surfaces that contain both those conics. Therefore, the union of the planes containing $\beta_1,\beta_2$ is a member of the pencil. This union of planes satisfies 
\begin{align}\label{eq: rank cone}
\mathrm{rank}\:\big(B(\gamma)+\lambda B(\delta)\big) = 2.
\end{align}
Saturating the corresponding ideal with respect to the ideal generated by $\lambda$, we retrieve the polynomial system $\mathcal J_3\subseteq \QQ[\gamma,\delta]$ which coincides with $I$, thereby recovering~\eqref{eq: two view ideal}. \label{enum:C}
\end{enumerate}

Based on~\eqref{eq: rank cone}, Quan derived by hand two linearly independent constraints satisfied by the conic multiview variety. These cut out the variety only locally. In contrast, our computational methods produce three constraints that generate its defining ideal; we have verified in \texttt{Macaulay2} that~\eqref{eq: two view ideal} is minimally generated by these three equations. Before the proof, we need a lemma.

\begin{lemma}\label{le: im of rat poly} Let $\varphi:\CC^{N_1}\to \CC^{N_2}$, sending $x\mapsto y=\varphi(x)$, be polynomial with rational coefficients. There is an ideal $I\subseteq \QQ[y]$ whose zero locus over $\CC$ is the closure of $\mathrm{Im}\: \varphi$. 

In addition, if $I$ is prime over $\QQ$, then it is radical over $\CC$. 
\end{lemma}

Suppose $K$ is a field and $f_1,\ldots,f_k\in K[z]$. We write $\langle f_1,\ldots,f_k\rangle_K$ for the ideal in $K[z]$ generated by $f_i$. We let $V_K(I)$ denote the zero locus of $I$ over the field $K$.

\begin{proof} For the fields $K\in \{\QQ,\CC\}$, consider the incidence ideals $I_{\mathrm{inc}}^K:=\langle y-\varphi(x)\rangle_K\subseteq K[x,y]$. Let $G_{\mathrm{inc}}^K$ be Gröbner bases over $K$ with respect to the monomial order 
\begin{align}
    x_1>\ldots>x_{N_1}>y_1>\ldots>y_{N_2}
\end{align} 
of $I_{\mathrm{inc}}^K$. We may without restriction assume that $G_{\mathrm{inc}}^{\QQ}\subseteq G_{\mathrm{inc}}^{\CC}$. By~\cite[Section 3.1, Theorem 2]{cox2007ideals}, $G^K:=G_{\mathrm{inc}}^K\cap K[y]$ is a Gröbner basis of $I^K:=I_{\mathrm{inc}}^K\cap K[y]$. Note that $G^{\QQ}\subseteq G^{\CC}$, implying the inclusion $V_\CC(I^{\CC})\subseteq V_\CC(I^{\QQ})$. Further, by~\cite[Section 3.3, Theorem 1]{cox2007ideals}, $V_K(I^{K})$ is the smallest variety containing $\varphi(K^{N_1})$. Since $\varphi(\QQ^{N_1})\subseteq \varphi (\CC^{N_1})$, this implies $V_\QQ(I^{\QQ})\subseteq V_\CC(I^{\CC})$. Noting that $V_\CC(I^{\QQ})$ is the smallest complex variety containing $V_\QQ(I^{\QQ})$, we conclude $V_\CC(I^\QQ)\subseteq V_\CC(I^\CC)$ and $V_\CC(I^\QQ)=V_\CC(I^\CC)$. Then $I=I^\QQ$ satisfies the statement.

The last part is due to~\cite[Lemma 10.43.6]{stacks-project}.    
\end{proof}

\begin{proof}[Proof of \Cref{thm: main two view}] We start by constructing an ideal in $\QQ[\gamma,\delta]$ whose zero locus over $\CC$ is $\mathcal C_{\Ca,2}$. By \Cref{enum:A} or \Cref{enum:B}, the zero locus of the ideal $I$ defined by~\eqref{eq: two view ideal} contains $\mathcal C_{\Ca,2}$ (it cannot be contained in the zero locus of the other component~\eqref{eq: union of planes}, because that equation does not involve $\gamma$). Alternatively, using \Cref{enum:C}, we conclude that $V_\CC(I)\supseteq \mathcal C_{\Ca,2}$ directly. 

By \Cref{le: im of rat poly}, in the primary decomposition of $I$, there is a component $J$ such that $V_\CC(J)= \mathcal C_{\Ca,2}$. By our \texttt{Macaulay2} computations, $I$ is prime over $\QQ$, meaning that $I=J$ and that $I$ is radical over $\CC$. 

The defining ideal of a multi-projective variety is the set of polynomials that vanish on it. By~\cite[Proposition 2.3]{agarwal2022atlas}, we are left to see that $J$ is saturated with respect to the irrelevant ideal
\begin{align}
    \langle \gamma_0,\ldots,\gamma_5 \rangle \cap \langle \delta_0,\ldots,\delta_5\rangle.
\end{align}
We have checked this in \texttt{Macaulay2}. 
\end{proof}

\begin{remark} The condition \eqref{eq: union of planes} is satisfied if and only if $\delta$ is a union of two lines. To see this, we first observe the identity
\begin{align}
    \delta^\top \nu_2^2(x)=x^\top\begin{bmatrix}
      \delta_0 & \delta_1/2 & \delta_2/2\\
      \delta_1/2 & \delta_3 & \delta_4/2\\
      \delta_2/2 & \delta_4/2 & \delta_5
    \end{bmatrix} x.
\end{align}
The curve defined by this polynomial is reducible if and only if the $3\times 3$ matrix has rank at most two. Its determinant is exactly \eqref{eq: union of planes}.
\end{remark}

\begin{lstlisting}[language=Macaulay2 , caption={\texttt{Macaulay2} code for computing the defining ideal of the two-camera conic multiview variety.} , label={fig: code} ]
-- Ring with gamma, delta, h and X variables
R = QQ[ga_0..ga_5, de_0..de_5, h_1..h_3][X_0..X_3]
coneSystem = {
    ga_0*X_1^2 + ga_1*X_1*X_2 + ga_2*X_1*X_3
      + ga_3*X_2^2 + ga_4*X_2*X_3 + ga_5*X_3^2,
    de_0*X_0^2 + de_1*X_0*X_1 + de_2*X_0*X_2
      + de_3*X_1^2 + de_4*X_1*X_2 + de_5*X_2^2
    }
restrictedSystem = apply(coneSystem, f -> sub(f, {
    X_0 => -h_1*X_1 - h_2*X_2 - h_3*X_3
    }))
restrictionMatrix = (coefficients restrictedSystem#0)_1 |
    (coefficients restrictedSystem#1)_1

-- Elimination ring
S = QQ[ga_0..ga_5, de_0..de_5, h_1..h_3]
J = eliminate(
    sub(minors(2, restrictionMatrix), S),
    {h_1, h_2, h_3}
    )
decompose J
\end{lstlisting}

\subsubsection{Computing the singular locus}

Before we describe the singular locus of the conic multiview variety given two cameras, we need the following remark.

\begin{remark} \label{re: union of planes}
The condition \eqref{eq: two view ideal} trivially holds if one of the rows is zero. Row one is zero if and only if $\gamma$ is either a double line or the union of two lines through the epipole $e_1^2=(0:0:1)$. Row two is zero if and only if $\delta$ is a double line or a union of two lines meeting at the epipole $e_2^1$. We prove the statement for the first row. Consider the equalities
\begin{align}\label{eq: gamma id}
\gamma_4^2=4\gamma_3\gamma_5,\quad \gamma_2\gamma_4=2\gamma_1\gamma_5,\quad \gamma_2^2=4\gamma_0\gamma_5.
\end{align}
The double line $(ux_0+vx_1+wx_2)^2$ corresponds to $(u^2:2uv:2uw:v^2:2vw:w^2)\in \PP^5$. If $\gamma_5\neq 0$, then we use \eqref{eq: gamma id} to substitute $\gamma_0,\gamma_1,\gamma_3$ to get
\begin{align}
    \gamma =(\gamma_2^2:2\gamma_2\gamma_4:2\gamma_2(2\gamma_5):\gamma_4^2:2\gamma_4(2\gamma_5):(2\gamma_5)^2),
\end{align}
meaning that $\gamma$ is a double line. If instead $\gamma_5=0$, then $\gamma_2=\gamma_4=0$. This implies that 
\begin{align}
    \gamma^\top \nu_2^2(x)=\gamma_0x_0^2+\gamma_1x_0x_1+\gamma_3x_1^2=0,
\end{align}
which has precisely two solutions in $(x_0:x_1)$. As $x_2$ is free to vary, we get two lines. They both meet at the epipole $e_1^2=(0:0:1)$.
\end{remark}

\begin{proposition}\label{prop: sing locus} Let $\mathcal C_{\Ca,2}$ be a conic multiview variety given two cameras of distinct centers, and let $(\gamma,\delta)\in \PP^5\times \PP^5$ denote its coordinates. The singular locus of $\mathcal C_{\Ca,2}$ consists of four components that are characterized by the following degeneracy conditions on the back-projected cones:
\begin{enumerate}
    \item $B(\gamma)$ is a union of two planes that both contain the center of the second camera (dimension 7).\label{enum: gamma union}
\item $B(\delta)$ is a union of two planes that both contain the center of the first camera (dimension 7).  \label{enum: delta union}
\item $B(\gamma)$ and $B(\delta)$ are unions of planes that share a common plane (dimension 5).\label{enum: common plane}
\item $B(\gamma)$ and $B(\delta)$ are both double planes (dimension 4). \label{enum: double}
\end{enumerate}
\end{proposition}

A geometric intuition for why these components are part of the singular locus is that in each case, the projection from the blowup 
\begin{align}
    \overline{\{(\beta, \Phi_{\Ca,2}^{\mathrm{Ch}}(\beta)): c_1,c_2\not\in\beta \}} \subseteq \mathrm{Ch}_{\mathrm{plane}}(2,\PP^3)\times \mathcal C_{\Ca,2}
\end{align}
onto the second factor fails to be locally injective. As an example, consider \Cref{enum: gamma union}. For a generic image tuple $(\gamma,\delta)$ in this component, $B(\gamma)$ is a generic union of two planes that both contain the center $c_2$ of the second camera, and $B(\delta)$ is a generic degree-two cone through $c_2$. Their intersection is a union of four lines through $c_2$, two in each plane of $B(\gamma)$. Let $\beta$ be any union of two such lines, one in each plane. Then $(\beta, (\gamma,\delta))$ lies in the blowup. Since there are multiple choices of such $\beta$, the projection onto the second factor is not locally injective. This fact, however, does not guarantee singularity. A general principle in this direction is van der Waerden's purity theorem \cite[Section 7.2.2, page 272]{2002algebraic}.     

\begin{proof} It suffices to check it for cameras $\Ca=(\begin{bmatrix}
    0 & I
\end{bmatrix},\begin{bmatrix}
    I & 0
\end{bmatrix})$, by the linear isomorphism of \Cref{re: lin iso}. By the fact that the ideal generated by \eqref{eq: two view ideal} is a defining ideal, we can compute the singular locus via the Jacobian condition. We work in affine coordinates $(\gamma,\delta)\in\CC^6\times\CC^6$ and obtain four components: two components of affine dimension nine, defined by
\begin{align}
    \langle \gamma_2,\gamma_4,\gamma_5\rangle
    \quad\text{and}\quad
    \langle \delta_0,\delta_1,\delta_2\rangle,
\end{align}
together with a seven-dimensional component $I_1$ and a six-dimensional component $I_2$. 

The ideal $ \langle \gamma_2,\gamma_4,\gamma_5\rangle$ was considered in \Cref{re: union of planes}; it corresponds to \Cref{enum: gamma union}. Similarly, $ \langle \delta_0,\delta_1,\delta_2\rangle $ corresponds to \Cref{enum: delta union}. 

The ideal $I_1$ (resp. $I_2$) has 14 (resp. 12) generators, and so we do not write them out explicitly. Instead, we checked in \texttt{Macaulay2} that the ideal $I_1$ (respectively $I_2$) vanishes on the variety described in \Cref{enum: common plane} (respectively \Cref{enum: double}). Similarly to the proof of \Cref{thm: main two view}, it follows that these are the defining ideals for the varieties in the list.      
\end{proof}


\subsection{Three-view geometry}\label{ss: three-view} Consider a camera arrangement $\Ca$ with centers $c_i$. A necessary condition for a tuple $\gamma$ of image curves to lie in the conic multiview variety is that the back-projected cones $B(\gamma_i)$ intersect in a conic $\beta$ and $(\beta,B(\gamma_i))\in \overline{\Gamma}_{c_i}$ for every $i$. We illustrate by a dimension count that, given three collinear centers, the conic multiview variety has more constraints.

\begin{example}\label{ex: collin cam} Suppose $c_1,c_2,c_3$ are distinct centers that lie on a line $E$. Let $\beta$ denote the double line of $E$. We consider the set of tuples $(B_1,B_2,B_3)$, where $B_i$ are cones through $c_i$ that meet in $\beta$, and calculate its dimension. That they meet in a double line rather than a reduced line means that the cones are tangent to each other along this line. For a generic plane $H$, cones $B_i$ are uniquely determined by the intersections $\beta_i:=H\cap B_i$. These conics $\beta_i$ must pass through $e:=H\cap E$ and, assuming that $B_i$ are irreducible, must all have a common tangent line at $e$. There are three degrees of freedom for each $\beta_i$ through $e$ with a fixed tangent. As the choice of tangent line has one degree of freedom, the dimension of the set is $1+3\cdot 3=10$. Since the conic multiview variety is eight-dimensional, the set of such cones is not contained in the multiview variety.   
\end{example}

A natural problem in multiview geometry is to study the number of image curves $\gamma_i$ that are involved in equations that cut out multiview varieties. The lower this number is, the simpler the constraints are. It is known that equations that involve three image points cut out the point multiview variety $\mathcal M_\Ca$, and as long as the centers are not coplanar, conditions that involve two image points suffice~\cite[Proposition 5]{trager2015joint}. For the line multiview variety, equations involving two image curves are not enough for any arrangement. Indeed, any two generic back-projected planes meet in a line away from the centers. The same holds for the conic multiview variety by taking two generic back-projected double planes instead. In the next example, we see that even if we assume that the back-projected cones are irreducible, equations involving two image curves are not enough to describe the conic multiview variety in three-view geometry.

\begin{example}\label{ex: pairwise} Assume that the centers $c_1,c_2,c_3$ span a plane $H$. In the affine chart defined by $H$, we may take back-projected cones $B_1,B_2,B_3$ to be cylinders. Assume they all have the same radii and their axes meet in a common point. Then, $B_1,B_2,B_3$ are irreducible and they meet pairwise in two irreducible conics, but $B_1\cap B_2\cap B_3$ does not contain a conic curve. This is illustrated in \Cref{fig: cylinders}.
\end{example}

\begin{figure}
    \begin{center}
\scalebox{1}[1]{\includegraphics[width = 0.43\textwidth]{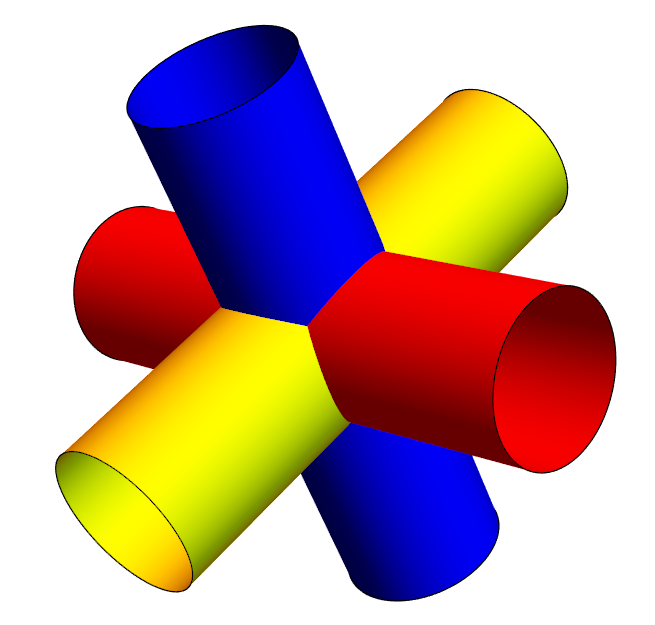}}
    \hspace{0.5cm}
\scalebox{1}[1]{\includegraphics[width = 0.43\textwidth]{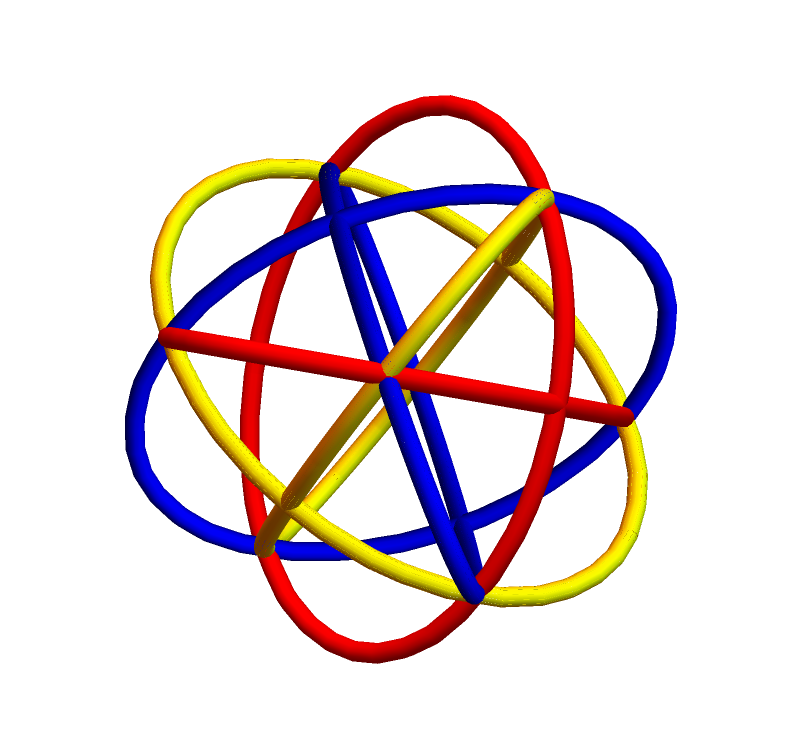}}
\end{center}
    \caption{Left: Three cylinders that intersect pairwise in the union of two irreducible conics. Right: these intersection curves have been plotted. We designate the intersection curves with the third color, i.e., the color that is not present among the intersecting cylinders.}
    \label{fig: cylinders}
\end{figure}


\subsection{Multiview geometry}\label{ss: n-view geo} Our main theorem on the $n$-view geometry for conic multiview varieties is the next result. We believe that the main ideas can be extended to higher degree plane curves, although this transition is not immediate. Let $\Ca$ be a camera arrangement of $n$ cameras with centers $c_i$.

\begin{theorem}\label{thm: set-theo} The conic multiview variety $\mathcal C_{\Ca,2}$ is cut out by the condition on image tuples $\gamma=(\gamma_1,\ldots,\gamma_n)$ that there is a conic $\beta$ in the intersection of all $B(\gamma_i)$ such that 
\begin{align}
    (\beta,B(\gamma_i))\in \overline{\Gamma}_{c_i} \textnormal{ for every }i=1,\ldots,n
\end{align}
if and only if all centers are distinct, no three of them lie on a line, and no eight of them lie on a conic.  
\end{theorem}

This theorem is the degree-two version of the main set-theoretic result on line multiview varieties \cite[Theorem 2.5]{breiding2023line}. Just as in the proof of that theorem, we work with structured sequences of world features. Here, surfaces with one family of conic curves and one family of lines replace the smooth quadric surfaces and their two families of lines used for line multiview varieties. A key insight of this paper is that it is not enough for the back-projected cones to meet in a conic. The cones also have to be compatible with respect to $\beta$, as captured by the blowups $\overline{\Gamma}_{c_i}$.

Before the proof, we need additional results concerning conics meeting prescribed points and lines:

\begin{lemma}\label{le: five p} Let $e_1,\ldots,e_5$ be five points in the plane. \textnormal{(1)} If no three are collinear, then there is a unique conic through each point, and this curve is irreducible. \textnormal{(2)} If no four are collinear, then there is a unique (and potentially reducible) conic through each point.
\end{lemma}

\begin{proof} See for instance \cite[Example 1.20 \& 1.21]{harris2013algebraic}.
\end{proof}

\begin{proposition}\label{prop: finitely} Let $X_1,X_2,\ldots,X_7\in \PP^3$ be distinct, coplanar, and not contained in a reducible conic. Let $L_2,\ldots,L_7$ be six generic lines in $\PP^3$ through $X_2,\ldots,X_7$, respectively. There are at most finitely many conic curves meeting $X_1,L_2,\ldots,L_7$.  
\end{proposition}

\begin{proof} The set $\mathcal H_{X_1}$ of planes through $X_1$ is two-dimensional. Fix a plane $H\in \mathcal H_{X_1}$, and define $Y_i:=L_i\cap H$. 

We first prove that at most one conic in $H$ can meet $X_1$ and all $L_i$. In this direction, observe that by the genericity of $L_2,\ldots,L_7$, to each subset of four lines among these, there are exactly two lines in $\PP^3$ that meet each of them~\cite[Section 3.4.1]{eisenbud20163264}. Further, as a consequence of genericity and the fact that $X_i$ do not lie on a union of two lines, all such $2\binom{6}{2}=30$ lines are distinct and meet only four of the $L_i$. Therefore, $H$ contains at most one tuple of four collinear points $Y_i$. Note that $H$ can contain at most one of the lines $L_i$. It is then possible to choose a subset of five points from $X_1,Y_i$ such that no four are collinear. This suffices by~\Cref{le: five p}.


We second prove that at most finitely many planes $H\in \mathcal H_{X_1}$ contain conic curves meeting $X_1$ and each $L_i$.
To proceed, let $H\in \mathcal H_{X_1}$ be generic and define $Y_i$ as above. By the genericity of the lines $L_i$, $Y_i$ can be seen as generic points in the plane $H$. Therefore, no six of the points $X_1,Y_i$ are contained in a conic. This shows that the set of planes $H\in \mathcal H_{X_1}$ containing conics that meet $X_1$ and a subset of five of the lines $L_2,\ldots,L_7$ is at most one-dimensional. Let $D_1^{(1)},\ldots,D_s^{(1)}$ (respectively $D_1^{(2)},\ldots,D_t^{(2)}$) denote the (at most one-dimensional) irreducible components of $H\in \mathcal H_{X_1}$ containing conics meeting $X_1$ and the five lines $L_2,\ldots,L_6$ (respectively $L_2,\ldots,L_5,L_7$). If $D_i^{(1)}$ and $D_j^{(2)}$ are one-dimensional and coincide for some $i,j$, then they would be independent of $L_6$ and $L_7$ given $L_2,\ldots,L_5$. This is, however, a direct contradiction. The intersection of two components $D_i^{(1)}$ and $D_j^{(2)}$ is therefore at most zero-dimensional. 

Since each $H$ corresponds to at most one conic, we are done.
\end{proof}

\begin{lemma}\label{le: eucl} Let $\mathcal N$ be an irreducible variety, and let $\mathcal L$ be a union of proper subvarieties. The Euclidean closure of $\mathcal N\setminus\mathcal L$ is $\mathcal N$.
\end{lemma}

\begin{proof} Consider the identity morphism $\mathrm{id}: \mathcal N\to \mathcal N$. The Zariski closure of the image of this map restricted to $\mathcal N\setminus \mathcal L$ is $\mathcal N$, since $\mathcal N$ is irreducible and $\mathcal L$ is a union of proper subvarieties. By Chevalley's theorem~\cite[Theorem 3.16]{JoeHarris}, the Zariski closure of its image equals its Euclidean closure. 
\end{proof}

\begin{theorem}\label{thm: seven lines} Let $X_1,\ldots,X_7\in \PP^3$ be distinct and contained in an irreducible conic $\beta$. Suppose that $L_1,\ldots,L_7$ are seven generic lines in $\PP^3$ through $X_1,\ldots,X_7$, respectively. Then, there is a sequence $\beta_m\to \beta$ in Euclidean topology such that $\beta_m$ meets each $L_i$ but none of $X_i$.
\end{theorem}

\begin{proof} Consider the variety 
\begin{align}\begin{aligned}
    \mathcal Y:=\{\beta':\beta' \textnormal{ meets each }L_1,\ldots,L_7\}\subseteq \mathrm{Ch}_{\mathrm{plane}}(2,\PP^3)\subseteq \PP^{20}. 
\end{aligned}
\end{align}
Note that $\mathcal Y$ is cut out by exactly seven linear equations, namely $\beta^\top \nu_2^5(L_i)=0 \quad\textnormal{ for }\quad i=1,\ldots,7.$ Since $\dim \mathrm{Ch}_{\mathrm{plane}}(2,\PP^3)=8$, the irreducible components of $\mathcal Y$ are at least one-dimensional.   

We deduce that $\beta$ is contained in an irreducible component of $\mathcal Y$ that is at least one-dimensional, implying by~\Cref{le: eucl} that there is a sequence $\beta_m\in \mathcal Y$ converging to $\beta$. Due to~\Cref{prop: finitely}, at most finitely many of $\beta_m$ can meet one of the points $X_1,\ldots,X_7$.
\end{proof}

\begin{proof}[Proof of \Cref{thm: set-theo}] Recall by \Cref{prop: dim} that the conic multiview variety is irreducible and of dimension eight as long as there are two cameras with distinct centers. Throughout the proof, we use \Cref{prop: blowup}.

$\Rightarrow)$ Assume without restriction that $c_1=\cdots =c_k$ for some $k\ge 2$ and that all other centers are distinct from these. We construct a tuple of back-projected cones $(B_1,\ldots,B_n)$ that satisfies the condition of the statement, but that does not lie in the conic multiview variety. Let $B_1,\ldots,B_k$ be any cones meeting in a double line $\beta$ through $c_1$ away from $c_{k+1},\ldots,c_n$. For $i\ge k+1$, define $B_i:=c_i\vee \beta$. This yields a tuple of cones that does not correspond to $\mathcal C_{\Ca,2}$. This is because in the image of $\Phi_{\Ca,2}$, $B_1,\ldots,B_k$ are always equal.

Suppose all centers are distinct. For a fixed line $E$, assume $c_i,i\in S,$ are the centers that lie on $E$. Assuming $|S|\ge 3$, we show a contradiction. We choose $B_i,i\in S,$ as in \Cref{ex: collin cam} such that they meet in the double line $\beta$ of $E$, and define $B_i:=c_i\vee \beta$ for all other indices $i$. A generic such tuple $(B_1,\ldots,B_n)$ does not lie on the conic multiview variety by the dimension count.

Suppose all centers are distinct, and no three lie on a line. Fix a conic curve $\beta$. Let $c_i,i\in S,$ be the centers that lie on $\beta$. If $|S|\ge 8$, then $\beta$ must be an irreducible conic curve. Let $B_i,i\in S,$ be generic unions of two planes such that one contains $\beta$ and the other is tangent to the curve. Then, $(\beta,B_i)\in \overline{\Gamma}_{c_i}$ for every $i\in S$. Define the remaining cones as $B_i:=c_i\vee \beta$. Since $|S|\ge 8$, there are at least eight dimensions of freedom in choosing $B_i,i\in S$. As this eight-dimensional set is not equal to the irreducible $\mathcal C_{\Ca,2}$, it cannot be a subvariety of it. 

$\Leftarrow)$ The conditions hold on the image $\mathrm{Im}\; \Phi_{\Ca,2}$, and it therefore suffices to prove that any image tuple $\gamma$ that satisfies the condition lies in the conic multiview variety. Let $\beta$ be a plane conic curve that is contained in the intersection of $B_i=B(\gamma_i)$ and satisfies $(\beta,B_i)\in \overline{\Gamma}_{c_i}$ for every $i$. In each case below, we find a sequence $\beta_m\to \beta$ meeting no centers such that $B_i$ is the limit of $c_i\vee \beta_m$. Observe that if $\beta$ does not meet any centers, then we can let $\beta_m=\beta$. If $\beta$ meets only one center, say $c_1$, then by the fact that $(\beta,B_1)\in \overline{\Gamma}_{c_1}$, there is a sequence $\beta_m\to \beta$ such that $B_1=\lim 
 c_1\vee \beta_m$. Since $B_i=\lim 
 c_i\vee \beta_m$ for $i\ge 2$, this sequence suffices.

In the case that $\beta$ is a double line, let $c_1,\ldots,c_k,k\le2,$ be the at most two centers that lie on $\beta$. The cases $k=0,1$ were handled above; we assume $k=2$. The cones $B_1,B_2$ meet in a degree-four curve. If both $B_1,B_2$ are irreducible, then they meet at most in the double line $\beta$ at $c_1\vee c_2$. They also meet in a conic curve $\beta'$ away from the centers. For a sequence $\beta_m'\to \beta'$ away from $c_1\vee c_2$, we write $B_{i,m}=c_i\vee \beta_m'$ for $i=1,2$. As $B_{1,m},B_{2,m}$ meet in degree-four curves, there is a sequence of plane conics $\beta_m$ distinct from $\beta_m'$ such that $B_{1,m}\cap B_{2,m}=\beta_m\cup \beta_m'$. Further, $\beta_m$ does not meet $c_1,c_2$. In particular, we have $B_{i,m}=c_i\vee \beta_m$, $\beta_m\to \beta$, and $B_{i,m}\to B_i$ for $i=1,2$. We also have $c_i\vee \beta_m\to B_i$ for $i>2$. 

If $B_1$ is irreducible, and $B_2$ is a union of two distinct planes, neither of which is tangent to $B_1$, let $\beta'$ be the union of two lines such that $B_1\cap B_2=\beta\cup \beta'$. The conic $\beta'$ does not meet $c_2$. Let $H$ be the unique hyperplane through $c_1$ such that $\beta'=H\cap B_1$. Let $H_m\to H$ be a sequence of hyperplanes away from $c_1$. Define $\beta_m':=H_m\cap B_1$ and observe that $\beta_m'
\to \beta'$. Write $B_{i,m}:=c_i\vee \beta_m'$ for $i=1,2$ and note that $B_{i,m}\to B_i$. Define $\beta_m$ such that $B_{1,m}\cap B_{2,m}=\beta_m\cup \beta_m'$. By construction, $c_i\vee \beta_m\to B_i$ for $i=1,\ldots,n$.

The remaining settings ($B_1$ is irreducible and $B_2$ is a union of two planes, at least one of which is tangent to $B_1$, or both cones are reducible) appear as limits of sequences of cones from the previous assumption. 

In the case that $\beta$ is a union of two distinct lines $E_1\cup E_2$ such that $E_1\cap E_2$ is not a center, let $c_1,\ldots,c_k,k\le 4,$ be the at most four centers that lie on $\beta$. The back-projected cones $B_1,\ldots,B_k$ must be unions of two planes. Let $H$ be the plane spanned by $E_1,E_2$. Then $B_i=H\cup P_i$ for planes $P_i$ and $i\le k$. Let $L_i$ be generic lines through $c_i$ inside $P_i$. Let also $L_0$ be a generic line through $E_1\cap E_2$. Following the theory of line multiview varieties~\cite{breiding2023line}, we can construct two sequences of lines $K_m\to E_1,K_m'\to E_2$ that meet in $L_0$ for each $m$, i.e., $\beta_m=K_m\cup K_m'$ are plane conics, such that the cones $c_i\vee \beta_m$ tend to $B_i$.

In the case that $\beta$ is a union $E_1\cup E_2$ of two distinct lines such that $E_1\cap E_2$ is a center, say $c_1$, let $c_1,\ldots,c_k,k\le 3,$ be the at most three centers that lie on $\beta$. It suffices to consider the case when $k=3$ and $c_2\in E_1,c_3\in E_2$. An image tuple satisfying the conditions has that $B_1$ is any cone containing $\beta$, while $B_i=H\cup P_i$ for $i=2,3$ are unions of planes, one being the unique plane $H$ containing $\beta$, and the other being a plane $P_i$ tangent to $\beta$ at $c_i$. It suffices to choose a generic such triplet. Let $L_i$ be a generic line in $P_i$ through $c_i$, and let $L,L',L''$ be generic lines in $B_1$. Define sequences $x_{i,m}\to c_i$ in $L_i$ for $i=2,3$ and $x_{1,m}\to c_1$ in $L$ away from the centers, and set $K_{i,m}:=c_1\vee x_{i,m}$. Let $B_{1,m}$ denote the cone defined uniquely by the containment of the five lines $K_{2,m},K_{3,m},L,L',L''$. Because $K_{2,m}\to E_1$ and $K_{3,m}\to E_2$, we have that $B_{1,m}\to B_1$. Define $H_m$ as the span of $x_{1,m},x_{2,m},x_{3,m}$. Further, set $\beta_m$ as the unique plane conic through $x_{1,m},x_{2,m},x_{3,m}$ and $x_m':=H_m\cap L', x_m'':=H_m\cap L''$. To see that $B_{i,m}:=c_i\vee \beta_m$ for $i=2,3$ converges to $B_i$, first note that $L_{i}\subseteq B_{i,m}$. The only plane through $L_i$ tangent to $\beta$ at $c_i$ is $P_i$, proving that $B_{i,m}\to H\cup P_i$. Finally, $\beta_m\to \beta$ because $\beta$ coincides with the intersection of $B_1,B_2,B_3$, where the limit must lie. 

In the case that $\beta$ is an irreducible conic, let $c_1,\ldots,c_k,k\le 7,$ be the at most seven centers that lie on $\beta$. The back-projected cones $B_i$ for $i=1,\ldots,k$ are by assumption unions of planes, one being the plane $H$ spanned by $\beta$, and the other being a plane $P_i$ tangent to $\beta$. It suffices to consider a generic such tuple. Take generic lines $L_i\subseteq P_i$ through $c_i$. There are at most seven such lines, and by \Cref{thm: seven lines} there is a one-dimensional family of conics $\beta_m\to \beta$ meeting each of these lines and no centers among $c_1,\ldots,c_k$. As is seen in the proof of \Cref{le: blowup}, $B_{i,m}:=c_i\vee \beta_m$ tend to $B_i$. 
\end{proof}





\section{Triangulation Complexity}\label{s: Tri Complexity}
In Structure-from-Motion, \textit{triangulation} is the recovery of world features given a known camera arrangement. Triangulation is a nearest point problem, mathematically formulated as follows in the general setting \cite{draisma2016euclidean}. For a variety $\mathcal{N}\subseteq \RR^m$ and a point $U\in\RR^m$ outside the variety, the nearest point problem is:
\begin{equation}\label{eq: ED_problem}
    \mathrm{minimize}\quad \sum_{i=1}^m(X_i-U_i)^2\quad \textnormal{subject to} \quad X\in \mathcal{N}\setminus \mathrm{Sing}(\mathcal N),
\end{equation}
where $\mathrm{Sing}(\mathcal N)$ is the singular locus of $\mathcal N$. This problem models the process of error correction and fitting noisy data to a mathematical model. The \textit{Euclidean distance degree} (EDD) of $\mathcal N$ is the number of complex solutions to the critical point equations associated to~\eqref{eq: ED_problem}, called \textit{ED-critical} points, for a given generic point $U\in \RR^m.$ For a variety $\mathcal N$ in a product of projective spaces, the EDD is the EDD of the affine variety we get by taking affine charts in each projective space. Here, we work with the standard affine patches $\gamma_0=1$ in the image curve spaces. The EDD is an estimate of how difficult it is to solve this problem by exact algebraic methods.

EDDs can be (in many cases confidently) conjectured by choosing a random data point $U$ and checking the degree of the critical ideal. Doing so for the variety generated by \eqref{eq: two view ideal} resulted in a program that did not terminate in reasonable time. Instead, based on numerical experiments using \texttt{HomotopyContinuation.jl}\cite{breiding2018homotopycontinuation}, we make the following conjecture:
\begin{conjecture} Given an arrangement of two generic cameras $\Ca$, the Euclidean distance degree of $\mathcal C_{\Ca,2}$ is $538$.
\end{conjecture}

We use a numerical method called \textit{monodromy} to solve for ED-critical points. Given a polynomial system $F=F_{u_0}$ belonging to a family of systems $\mathcal F=\{F_u:u\in \CC^k\}$ and a start solution $x$ such that $F(x)=0$, monodromy uses homotopy continuation to track $x$ as $u_0$ varies in a loop that starts and ends at $u_0$. Typically, the tracked solution is distinct from the original one. Monodromy has been successfully used in other algebraic vision settings~\cite{duff2022galois}. In our case, the family of systems would naturally be parametrized by two cameras $C_1,C_2$. In order to reduce the number of parameters, we fix cameras to be of the form
\begin{align}
    C_1= \begin{bmatrix}
        1 & 0 & 0 & s_1\\
        0 & 1 & 0 & s_2\\
        0 & 0 & 1 & s_3
    \end{bmatrix} , \quad C_2= \begin{bmatrix}
        1 & 0 & 0 & t_1\\
        0 & 1 & 0 & t_2\\
        0 & 0 & 1 & t_3
    \end{bmatrix}.
\end{align}
To account for the fact that the cameras are specialized, we compute the \textit{generic} EDD, which means that we multiply each term of \eqref{eq: ED_problem} by a generic real number. We use the parametrization map $\Phi_{\Ca,2}$ on an affine patch $\CC^8=\CC^3\times \CC^5$ of $(\PP^3)^*\times \PP^5$. The number of ED-critical points computed by monodromy in this setup is $1076 = 2 \cdot 538$. However, the map $\Phi_{\Ca,2}$ is generically 2-to-1 as seen in \Cref{fig: conics}. 
 
Based on previous work, such as \cite{harris2018chern,EDDegree_point,duff2024metric}, we expect the number of ED-critical points given $n$ generic cameras and generic data to grow as a degree-eight polynomial in $n$. Due to the size of the systems involved, we have not been able to verify this numerically.



\bibliographystyle{alpha}
\bibliography{VisionBib}




\end{document}